\documentclass[reqno,11pt]{amsart} 

\usepackage[foot]{amsaddr}
\usepackage[utf8]{inputenc}
\usepackage[english]{babel}
\usepackage{lipsum}
\usepackage[margin=3cm]{geometry}

\usepackage{enumitem}

\setlist{nolistsep}
\usepackage{mathtools}
\usepackage{mathrsfs}
\usepackage{lettrine}
\usepackage{mfirstuc}
\usepackage{multicol}
\usepackage{multirow}
\usepackage{makecell}
\usepackage[table]{xcolor}

\usepackage[overload]{textcase}
\usepackage{tasks}
\usepackage{bbm}
\usepackage{amsmath}
\usepackage{amssymb}
\usepackage{amsfonts}
\usepackage{amsthm}
\usepackage{dsfont}

\usepackage{graphicx}
\usepackage{wrapfig}
\usepackage[labelfont={bf,it},textfont=it]{caption}
\usepackage[dvipsnames]{xcolor}
\usepackage[pagebackref=true,colorlinks=true, linkcolor=Blue, citecolor=Blue]{hyperref}

\usepackage{widetable}

\usepackage[scr=pxtx, scrscaled=1]{mathalfa}
\usepackage{xfrac}
\usepackage{soul}
\usepackage{tikz}
\usepackage{setspace}
\usepackage{yfonts}
\usepackage{bm}

\newcommand{\abs}[1]{\lvert#1\rvert}

\newcommand*{\R}{\mathbb{R}}

\newcommand*{\Z}{\mathbb{Z}}
\newcommand*{\K}{\mathcal{K}}
\newcommand*{\N}{\mathbb{N}}

\newcommand{\BM}{\text{BM}}
\newcommand{\Leb}{\text{Leb}}
\newcommand{\bfPi}{\boldsymbol\Pi}
\newcommand{\bfeta}{\boldsymbol\eta}
\newcommand{\bfs}{\boldsymbol\s}
\newcommand{\bfzeta}{\boldsymbol\zeta}
\newcommand{\bfxi}{\boldsymbol{\xi}}
\newcommand{\bfx}{\boldsymbol x}
\newcommand{\bfy}{\boldsymbol y}
\newcommand{\bfo}{\boldsymbol o}
\renewcommand*{\d}{\mathrm{d}}
\newcommand*{\e}{\mathrm{e}}

\renewcommand*{\Xi}{\varXi}
\renewcommand*{\epsilon}{\varepsilon}
\renewcommand*{\theta}{\vartheta}
\renewcommand*{\Theta}{\varTheta}
\renewcommand*{\L}{\varLambda}
\renewcommand*{\Delta}{\varDelta}

\newcommand*{\1}{\mathds{1}}
\newcommand{\ssup}[1] {{{\scriptscriptstyle{({#1}})}}} 
\newcommand{\tsup}[1] {{{\scriptscriptstyle{[{#1}]}}}} 
\renewcommand{\o}{\omega}
\renewcommand{\O}{\Omega}

\newcommand{\qc}{q_{\rm c}}
\newcommand{\s}{\sigma}
\newcommand{\g}{\gamma}
\newcommand{\eps}{\varepsilon}
\newcommand{\F}{\mathcal F}
\newcommand{\one}{\mathds 1}

\newcommand*{\Cov}{\mathrm{Cov}}
\renewcommand{\ssup}[1] {{{\scriptscriptstyle{({#1}})}}}

\newtheorem{theorem}{theorem}[section]

\newtheorem{lemma}[theorem]{Lemma}
\newtheorem{proposition}[theorem]{Proposition}
\newtheorem{thm}[theorem]{Theorem}
\newtheorem{definition}[theorem]{\bf Definition}

\theoremstyle{remark}
\newtheorem{remark}[theorem]{Remark}

\usepackage{textcase}
\usepackage[nopatch=footnote]{microtype}
\usepackage{xcolor}

\title[Locality properties for Widom--Rowlinson models in random environments]{Locality properties for discrete and continuum Widom--Rowlinson models in random environments}

\author{Benedikt Jahnel$^1$}
\address{$^1$Technische Universit\"at Braunschweig and Weierstrass Institute Berlin, {\tt benedikt.jahnel@tu-braunschweig.de}}
\author{Christof K\"ulske$^2$}
\address{$^2$Ruhr University Bochum, {\tt christof.kuelske@ruhr-uni-bochum.de}}
\author{Alexander Zass$^3$}
\address{$^3$Weierstrass Institute Berlin, {\tt zass@wias-berlin.de}}

\begin{document}

\begin{abstract}
We consider the Widom--Rowlinson model in which hard balls of two possible colors are constrained to a hard-core repulsion between particles of different colors, in quenched random environments. These random environments model spatially dependent preferences for the attachment of balls. We investigate the possibility to represent the joint process of environment and infinite-volume Widom--Rowlinson measure in terms of continuous
(quasilocal) Papangelou intensities. We show that this is not always possible: In the case of the symmetric Widom-Rowlinson model on a non-percolating environment, we can explicitly construct a discontinuity coming from the environment.
This is a new phenomenon for systems of continuous particles, but it can be understood as a continuous-space echo of a simpler non-locality phenomenon known to appear for the diluted Ising model (Griffiths singularity random field~\cite{grising_2000}) on the lattice, as we explain in the course of the proof.
\medbreak\noindent
\emph{\keywordsname:} Gibbs measures; quasilocality; point processes; Widom--Rowlinson models; Papangelou intensities; disordered systems.
\end{abstract}

\maketitle

\section{Introduction}\label{sec_introduction}
In this article we consider the Widom--Rowlinson model (WRM) in a random environment, from the discrete lattice to the continuum. Physically the model describes a quenched random substrate on which particles of two types 
can attach, subject to a repulsive interaction between particles of different types. The attachment process is described in terms of a WRM in each of the connected components of the substrate. In the classical WRM, first introduced in~\cite{WR70},
the substrate (or base-space) is the Euclidean space $\R^d$; in this case the particle positions are drawn from two independent homogeneous Poisson point processes (PPP) of intensity $\lambda_+$ for the plus-particles, and  $\lambda_-$ for the minus-particles, under the constraint that particles of different types cannot be closer than a fixed radius. Recall that, as a fundamental property, the symmetric model with large and equal particle intensities $\lambda:=\lambda_+=\lambda_-$ shows a phase transition: it has two translation-invariant infinite-volume Gibbs measures $\mu^+$ and $\mu^-$ which show symmetry breaking in the particle densities for plus-particles and minus-particles, see~\cite{ruelle1971existence,bricmont1984structure,chayes1995analysis,dereudre2019phase}.  
More generally the WRM, and its high-density phase transition, has also been studied for spatially-homogeneous \emph{discrete base spaces}, including the integer lattice~\cite{higuchi_2004} or the Cayley tree, as well as in versions where the strict hardcore interaction is replaced by a mollified interaction, see~\cite{kissel2019hard,kulske2019gibbs}. For dynamical questions of metastability, see~\cite{den2019widom}. 

In the present work on the WRM in \emph{quenched environments}, we are after structural locality properties of the \emph{joint measures} that describe the environment process together with the process of attached WR particles. The environment processes in the present work are given by percolation models driven by spatially independent processes, such as the Poisson--Boolean model (PBM) in the continuum, the Poisson--Gilbert graph (PGG), or the clusters of iid Bernoulli site percolation (BSP) on the lattice. 

We are ultimately interested in exhibiting a discontinuity in the fully-continuous model and explore this question also in the two more discrete cases, which are interesting models in and of themselves. Indeed, not only does their study provide insight into the techniques and mechanisms needed for the case of the PBM, but they present interesting different features, see Table~\ref{tab:results}. Let us comment on these models individually. 

\medskip
{\bf Poisson--Boolean environments.} In the fully-continuous model the substrate is modeled by a PBM. By definition 
this model describes the distribution of a random subset of Euclidean space $A\subset \R^d$, obtained as the union of balls of independent random radii drawn from a distribution $\nu$, whose midpoints are chosen according to a driving PPP in Euclidean space with homogeneous intensity $\beta>0$. 
For precise definitions, see Section~\ref{sec:PBM}, Formula~\eqref{eq:PBM:joint} in our paper; for percolation properties of this model, see~\cite{meester1996continuum}. Conditional on a fixed realization of this environment, the positions of WR particles and their marks (plus or minus) are then distributed  according to the WRM on $A$, with symmetric intensity $\lambda$. This model is the most complex we consider in this work, 
and so it is worthwhile to explain setup, questions, and mechanisms in the limiting case of a substrate of zero radii. The meaningful limiting formulation of our problem for the PBM is the following. 

\medskip
{\bf Poisson--Gilbert graphs.} The environment is completely described by Poisson points in Euclidean space of a given intensity $\beta$. Each of these environment points serves as a potential docking site for a plus- or minus-WR particle, but it may also be left empty. This happens with site-wise independent a-priori probabilities $p_+,p_-,p_0$, subject to the WR hardcore condition that inter-particle distances must be greater or equal than $1$ between particles of different signs. 
We note that the resulting joint process is then a dependently {\em marked point process}. 
Its configuration space is the space of marked point clouds $(\eta, \sigma_\eta)$ 
where $\eta\subset \R^d$ is a locally finite point cloud, and the marks $\sigma_y$ take values in the mark space $\{-1,0,1\}$, where $0$ stands for a vacancy.  
In our treatment of the joint measures of the WRM on the full Poisson--Boolean environment we adopt a similar view as marked point process, where necessarily 
our spaces will then carry more information, see~\eqref{eq:PGG:joint}. 
Our analysis will start from the following further discretization of the model.

\medskip
{\bf Bernoulli site percolation.} In this setting the environment is described by a subcritical Bernoulli($q$) lattice field on $\Z^d$, i.e., each site $i\in\Z^d$ is either occupied ($v$, with probability $q$) or unoccupied ($u$, with probability $1-q$), independently of the others. As in the PGG setting, each occupied environment point is a potential docking site for a WR particle plus, minus, or zero. The joint measure of the WRM on this dilute lattice can then be seen as a finite-alphabet model with alphabet given by $\{u,-1,0,+1\}$.

\medskip
{\bf Papangelou intensities and specification kernels.} 
PPPs and marked point processes are most conveniently described in stochastic analysis in terms of {\em Papangelou intensities} $\rho({\bfx, \bfeta})$, see~\cite{dereudre2019introduction}, by means of the GNZ equation, see~\eqref{def:Pap}. In a dynamical interpretation, the GNZ equation 
describes the invariance of an infinite-volume measure under the dynamics which adds a marked particle $\bfx$ with rate $\rho({\bfx, \bfeta})$ to a given marked point cloud $\bfeta$ sampled from the measure, and kills particles with rate $1$, see~\cite{georgii2005conditional}. For readers familiar with infinite-volume equilibrium statistical mechanics, the GNZ equation is a spatially infinitesimal version of the DLR equation for models with discrete base spaces which are described via specifications. The Papangelou intensity then serves as an analogue to the single-site specification kernel $\gamma_{\bfx}$. For models in continuous space, specifications, i.e., families of kernels indexed by general bounded subvolumes 
in $\R^d$, are still very relevant. For this reason, it is useful to note that there is a simple relation between specifications and the Papangelou intensity given by formula~\eqref{Pap-DLR}. It indicates that all specification kernels (i.e., conditional probabilities for the process in bounded volumes $\Lambda$ of Euclidean space, given the realization of the process outside of $\Lambda$) can be derived from the simpler Papangelou intensity, by adding points iteratively. As Papangelou intensities refer to a single points $\bfx$ instead of nontrivial volumes $\Lambda$, they are very convenient for computations and analysis, and we prefer to focus on them in the present work. 

\medskip
{\bf Do our models have continuous Papangelou intensities?} 
In this paper we investigate the possibility to represent the joint process of environment and infinite-volume WRMs in terms of continuous (quasilocal) Papangelou intensities. This question is driven by similar investigations on the quasilocal Gibbsianness of a variety of models, mostly in the context of discrete base space models, see for example~\cite{bricmont1998renormalization,van2002possible,KNR04}. This is relevant also because non-locality of rates makes it impossible to define infinite-volume dynamics using standard techniques~\cite{liggett1985interacting}. For studies which consider transformations 
of the WRMs under time-evolutions, see~\cite{JK17} in the continuum, or~\cite{kissel2018dynamical,bergmann2023dynamical} in discrete setups. For a dynamic version of the WRM as a birth and death process and associated out-of-equilibrium questions, see~\cite{den2019widom} and upcoming works by the same authors.

\medskip
{\bf Main results.} Let us explain our result with respect to the fully-continuous model, where WR points are bound to lie on the area covered by the PBM describing the environment. We prove the following negative result: {\em The symmetric high-density WRM on a subcritical PBM does {\bf not} admit a continuous Papangelou intensity.}

This non-representability statement of joint measures in terms of good Papangelou intensities is prototypical for the other cases we treat, albeit modulo some notable changes, concerning the local and non-local nature of possible discontinuities, see Definition~\ref{def:PBM-locrob}.
In Table~\ref{tab:results} we present an overview of the results.
\rowcolors{2}{gray!25}{white}
\begin{table}
    \begin{tabular}{c|c|c|c}
    \rowcolor{gray!50}
        \multicolumn{1}{c|}{}& Section \ref{sec:DL} & Section \ref{sec:PGG} & Section \ref{sec:PBM}\\
        Environment & BSP($q$) & PGG($\beta$) in $\R^d$ & PBM($\beta$) w/random radii\\
        Small WR intensity & ql-Gibbs & non-ql-Gibbs & ql-Gibbs if bounded radii \\
        Large WR intensity & non-ql-Gibbs & non-ql-Gibbs & non-ql-Gibbs\\
    \end{tabular}
    \caption{Overview of the results on existence and absence of quasilocal (ql) Papangelou intensities for the joint measure of the WRM in three environments. The environments are given by Bernoulli site percolation (BSP), Poisson--Gilbert graphs (PGG), and Poisson--Boolean models (PBM) in the subcritical percolation regimes with respect to their corresponding parameters $q\in[0,1]$ and $\beta>0$.}
    \label{tab:results}
\end{table}
Let us highlight that the most notable qualitative difference between the three models is that for the environments given by the BSP and the PBM, we are able to show that the models undergo a \emph{phase transition} depending on the WR intensity, in the sense that, for small intensities, we can exhibit a continuous Papangelou intensity, while for large intensities this is not possible. Indeed, in the PGG, a different behavior emerges, as non-quasilocality holds for all intensities. This is a true continuum phenomenon created by the possible accumulation of environment points together with the fact that the environment is non-spatial from the point of view of the WRM, see Step~2 of the proof of Theorem~\ref{thm_2}.

Let us further note that in the PBM setting, natural non-equivalent families of Papangelou intensities exist, based on different a-priori measures and the necessarily different association rules for the WR particles, see~\eqref{eq:PBM:pap} and~\eqref{eq:PBM:pap3} and the respective Figures~\ref{fig:PBM:pap} and~\ref{fig:PBM:pap2}.
In any case, the discontinuity of the Papangelou intensity we encounter is due to intrinsic non-local effects in the joint measure, which occur when we are connecting or disconnecting large but finite clusters of the substrate, at large distances. 
These effects are persistent, and can not be remedied by redefinition of topologies on point clouds, see Definition~\ref{def:PBM-locrob}. 
Indeed, even though in the subcritical regime very large clusters are clearly very improbable, their effect cannot be neglected as it leads to a non-removable discontinuity. Moreover, we stress that this phase transition is not simply due to some unbounded nature of the interaction: it is still present even if the radii of the environment are bounded, i.e, when the interaction is completely finite range. This is a phenomenon reminiscent of Griffiths singularities~\cite{griffiths1969nonanalytic}, and is related to a (simpler) non-locality phenomenon that has been proved to occur for the Ising model on site percolation clusters on the lattice, see~\cite{grising_2000}. 
Our actual proof follows the idea that connecting large clusters far away creates discontinuity, 
which is however already in the lattice situation much more involved than for an Ising model, 
and relies on exploiting stochastic domination and percolation arguments for WRMs. 

To summarize, our results show that randomness of the environment induces non-local behavior in the joint measures, in ways which depend on the type of environment distribution.  Non-localities at large WR intensities occur for all types of environments. Broadly speaking and ignoring any technical challenges of continuous particle systems, this might have been a bold guess from the beginning for a reader knowing the Griffiths field. A bigger surprise is the special role of the PGG compared to the other two types of environments, as it displays discontinuities also in the region of small WR intensities. Our analysis reveals that this can be understood by the relevance of the network structure a PGG realization induces, and the information transportation over long distances in the low-density WR model on high-density networks.

\medskip
The manuscript is organized as follows. It proceeds from simpler environment processes to environments with higher complexity. In Section~\ref{sec:DL} we investigate the quenched WRM on the clusters of subcritical BSP and prove absence of quasilocal specifications due to an essential discontinuity point along connected halfspaces. In Section~\ref{sec:PGG} we consider the discrete WRM on the subcritical PGG and prove absence of quasilocality on the level of graph Papangelou intensities. Finally, in Section~\ref{sec:PBM} we show the absence of quasilocal Papangelou intensities for the quenched continuum WRM on the PBM. All proofs are presented in Section~\ref{sec_proofs}. Finally in the appendix in Section~\ref{sec:ap} we elaborate on an alternative approach towards the canonical Papangelou intensity of the continuum WRM on the PBM. 

\section{The WRM on the dilute lattice}\label{sec:DL}
Consider the \emph{Bernoulli lattice field} $\pi(\d\eta)$ on $\Z^d$, $d\ge 2$, with density $q\in(0,1)$, i.e., $\eta_i\in\{u,v\}$ where $\eta_i=v$ represents that the site $i\in \Z^d$ is occupied, $\eta_i=u$ represents that the site $i\in \Z^d$ is unoccupied and occupation at $i\in\Z^d$ happens independently with probability $q$.  
It will be convenient to occasionally identify $\eta\in \{u,v\}^{\Z^d}$ with the set of occupied sites, i.e., consider $\eta\subset\Z^d$. It is a well-known fact that there exists a critical threshold $0<\qc<1$ such that for all $q<\qc$ the field $\pi$ is in the subcritical percolation regime, where almost surely no infinite connected component of occupied sites exists~\cite{grimmett1999percolation}. In what follows we will always assume that $q<\qc$. 
We are interested in the joint measure
\begin{equation}\label{eq:DL:joint}
	K(\d\eta,\d\sigma_\eta) = \pi(\d\eta)\mu_\eta(\d\s_\eta),
\end{equation}
where $\mu_\eta$ is the countable product measure over the almost-surely finite components $(\eta^{\ssup{k}})_{k\in I}$ of $\eta$, i.e.,  
 $\mu_{\eta}(\d\sigma_{\eta})=\Pi_{k\in I}\mu_{\eta^{\ssup{k}}}(\d\sigma_{\eta^{\ssup{k}}})$. Then, for a finite component $\eta^{\ssup{k}}$, $\mu_{\eta^{\ssup{k}}}(\d\sigma)$ is the {\em discrete Widom--Rowlinson model} (WRM)~\cite{higuchi_2004}, i.e.,
\begin{equation}\label{eq:Disc_WRM}
	\mu_{\eta^{\ssup{k}}}(\sigma_{\eta^{\ssup{k}}}) = \frac{1}{Z_{\eta^{\ssup{k}}}}\ \prod_{i\in \eta^{\ssup{k}}}p_{\sigma_i}\1\{\s_{\eta^{\ssup{k}}} \text{ is feasible}\},\qquad \s_{\eta^{\ssup{k}}}\in \{-1,0,1\}^{\eta^{\ssup{k}}},
\end{equation}
where $p_0,p_{+},p_{-}\in(0,1)$ with $p_0+p_+ + p_-=1$ are the spin probabilities and feasibility
\begin{equation*}
    \1\{\s_{\eta^{\ssup{k}}} \text{ is feasible}\}=\prod_{i, j\in \eta^{\ssup{k}}\colon \abs{i-j}=1}\1\{\sigma_i\sigma_j\neq -1\}
\end{equation*}
represents the constraint that no neighboring spins face in opposite directions. $Z_{\eta^{\ssup{k}}}$ is the usual normalizing constant. In words, we consider the WRM with open boundary conditions on the finite clusters of a subcritical Bernoulli site percolation model and the measure $K$ can then be seen as a probability measure on $\O:=\{u,-1,0,1\}^{\Z^d}$ as illustrated in Figure~\ref{fig:DL:clusters}. In particular, 
$K(\d \bfeta)$ is a finite-alphabet model with alphabet $\{u,-1,0,1\}$ and a configuration $\O\ni \bfeta=(\eta,\s_\eta)$ is completely determined by its {\em gray configuration} $\eta$ (the set of occupied sites) together with its {\em coloring} $\s_\eta$.

Note that, in order for $K$ to be well-defined, we need to ensure measurability of  $\mu_\eta$ with respect to $\eta$. For this, note that we can define $\mu_\eta$ 
also in a super-critical regime for example as the extremal measure $\mu_\eta^+$ (with plus-boundary condition), whose existence follows by FKG-type arguments for monotone finite-volume limits on increasing observables. Furthermore, $\mu_\eta$ can be uniquely identified as a measure on $\O$ (without reference to $\eta$) via $\mu_\eta\otimes \delta_{u_{\eta^c}}$, where $u_{\eta^c}$ is the all-$u$ configuration on $\Z^d\setminus \eta$. Then, the measurablity of $\eta\mapsto\mu_\eta$ can be checked on local configurations $w_\L:=\{w'\in \O\colon w'_{\L}=w_{\L}\}$ and follows from the fact that $\mu_\eta(w_\L)$ is the limit of local (and therefore measurable) functions. 
\begin{figure}
\begin{minipage}{.6\textwidth}
	\begin{center}
		\includegraphics[width=.8\textwidth]{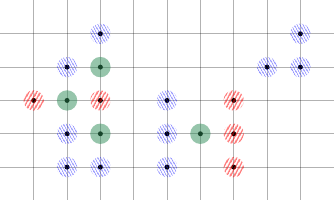}
	\end{center}	
\end{minipage}%
\begin{minipage}{.4\textwidth}
	\caption{Configuration of the WRM on the diluted lattice. Black vertices represent occupation in the underlying Bernoulli site percolation. Blue (resp.~green, red) disks represent the $+$ (resp.~$0$, $-$) spins in the WRM.}\label{fig:DL:clusters}
\end{minipage}
\end{figure}

As first noted in~\cite{grising_2000}, for the case of the Ising model defined on the clusters of a Bernoulli lattice field, the model $K$ seems very local at a first glance. However, this is not the case, when locality is defined in the following sense. 

Consider a {\em specification} $\g=(\g_\L)_{\L\Subset\Z^d}$ indexed by finite subsets $\L\Subset\Z^d$, i.e., a consistent and proper family of probability kernels $\g_\L(\bfeta_\L|\bfeta_{\L^{\rm c}})$ on $\O_\L:=\{u,-1,0,1\}^\L$ with boundary condition $\bfeta_{\L^{\rm c}}\in \O_{\L^c}$. Here $\L^{\rm c}=\Z^d\setminus\L$. We are interested in continuity properties of $\g$ with respect to the boundary condition based on the product topology on $\O$, where we equip $\O$ with the Borel-$\s$-algebra $\F$ associated to the product topology.  
\begin{definition}[Quasilocal specifications]
	Let $\gamma$ be a specification on $\O$. A configuration $\bfeta\in\O$ is called \emph{good} for $\g$ if, for any $\L\Subset\Z^d$ and any local function $f\colon \O\to\R$,
	\begin{equation*}
		\sup_{\bfeta'\colon \bfeta'_\Delta=\bfeta_\Delta}\abs{\g_{\L}(f\vert\bfeta'_{\L^c}) - \g_{\L}(f\vert\bfeta_{\L^c})}\to 0\qquad \text{ as } \Z^d\Supset\Delta\uparrow\Z^d. 
	\end{equation*}
	We denote by $\O(\g)$ the set of good configurations for $\g$, while elements of $\O\setminus\O(\g)$ are called \emph{bad} for $\g$. We say that $\g$ is \emph{quasilocal} if $\O(\g)=\O$.
\end{definition}
Let us note that quasilocality of specifications is in one-to-one correspondence to the existence of summable potentials such that the specification can be written as a Boltzmann weight with respect to this specification~\cite{kozlov1974gibbs,sullivan1973potentials}.

Based on a notion of quasilocal specifications we can now characterize quasilocal Gibbsianness. 
\begin{definition}[Quasilocal Gibbsianness]
	A probability measure $P$ on $\O$ is said to be \emph{quasilocally Gibbs}, if $P$ satisfies the DLR equations with respect to some quasilocal specification $\g$, i.e.,
	\begin{equation*}
		P(A)=\int P(\d \bfeta) \g_\L(A\vert\bfeta_{\L^c})\qquad \text{ for all $\L\Subset\Z^d$ and $A\in \F$}.
	\end{equation*}
	Otherwise $P$ is called \emph{non-quasilocally Gibbs}. 
\end{definition}
Let us note that in~\cite{sokal1981existence} it is shown that, under general conditions, for every random field there exists a specification satisfying the DLR equation. However, the specification is not necessarily quasilocal.  
Using these definitions, we can now formulate our first main result on the presence and absence of quasilocal Gibbsianess. 
\begin{thm}\label{thm_1}
Let $q<\qc$ and consider the symmetric WRM, where $p_+=p_-$ is sufficiently close to $1/2$. Then, the measure $K$ is non-quasilocally Gibbs. On the other hand, for $q<\qc$ and all sufficiently small $p_+,p_-$, $K$ is quasilocally Gibbs.
\end{thm}
As a point of reference, for the following sections, let us mention that the proof of Theorem~\ref{thm_1}, which is presented together with all other proofs in Section~\ref{sec_proofs}, is based on the canonical version of the specification given by $K(\bfeta_\L|\bfeta_{\L^{\rm c}})$. More precisely, leveraging factorization and translation invariance, rather than $K(\bfeta_\L|\bfeta_{\L^{\rm c}})$, it suffices for us to work with the following quantity,  
\begin{equation}
\label{eq:Disc_Pap}
\begin{split}
\rho_o(\bfeta_o,\bfeta_{o^{\rm c}}):=\frac{K(\bfeta_o|\bfeta_{o^{\rm c}})}{K(u_o|\bfeta_{o^{\rm c}})}
=\frac{q}{1-q}\frac{Z_{\eta_{\tsup{o}}}}{Z_{\eta_{\tsup{o}}\cup\{o\}}}\one\{\s_{{\rm n}(o,\eta)\cup\{o\}}\text{ is feasible}\}\one\{\eta_{\tsup{o}}\text{ is finite}\}
\end{split}
\end{equation}
where $\bfeta_o\in \{-1,0,1\}$, ${\rm n}(x,\eta)\subset\eta$ denotes the set of vertices in $\eta$ that are neighbors of $x$, and $\eta_{\tsup{x}}\cup\{x\}$ denotes the cluster of $\eta\cup\{x\}$ at $x$, where $\eta_{\tsup{x}}=(\eta_{\tsup{x}}\cup\{x\})\backslash\{x\}$ is the finite union of connected components of $\eta$ which are merged by adding $x$. As will become clear later, the quantity~\eqref{eq:Disc_Pap} can be seen as a {\em discrete Papangelou intensity}.

Let us also note that the last statement about quasilocal Gibbsianness in Theorem~\ref{thm_1} in fact holds for all $q\in [0,1]$. However, in order to make this precise we would have to define WRMs on infinite clusters, as explained above.

\section{The WRM on the Poisson--Gilbert graph}\label{sec:PGG}
In this section we investigate an intermediate model in which the diluted lattice is replaced by a subcritical Poisson--Gilbert graph (PGG). That is, we consider a homogeneous Poisson point process (PPP) $\Pi(\d \eta)$ in $\R^d$ with intensity $\beta>0$ and use the associated Gilbert graph $g(\eta)$ in which every pair of points $X_i,X_j\in \eta$ is connected by an undirected edge if and only if $|X_i-X_j|<1$. In order to introduce some notation, recall that $\Pi$  is a probability measure on the set $\mathcal X:=\{\eta\subset \R^d\colon |\eta_\L|<\infty\text{ for all bounded }\L\subset\R^d\}$ of locally finite subsets of $\R^d$, which features total spatial independence. Here $\eta_\Lambda$ denotes the restriction of $\eta$ to the volume $\Lambda\subset\R^d$. It is a well-known fact from continuum percolation theory that there exists a critical intensity $0<\beta_{\rm c}<\infty$ such that for $\beta<\beta_{\rm c}$ the PGG is subcritical in the sense that with probability one, $\Pi$ features only finite connected components in $g(\eta)$. $\beta$ hence plays the role of $q$ from Section~\ref{sec:DL} and we assume the Gilbert graph to be subcritical. 
We are interested in the joint measure
\begin{equation}\label{eq:PGG:joint}
	\K(\d\eta,\d\sigma) = \Pi(\d\eta)\mu_\eta(\d\s_\eta),
\end{equation}
where, for 
$\Pi$-a.e.~pointcloud $\eta$, the PGG decomposes into countably many  connected components with vertex sets $(\eta^{\ssup{k}})_{k\in I}$, and $\mu_{\eta}$ is the countable product measure over these components, $\mu_{\eta}(\d\sigma_{\eta})=\Pi_{k\in I}\mu_{\eta^{\ssup{k}}}(\d\sigma_{\eta^{\ssup{k}}})$, where for a finite component $\eta^{\ssup{k}}$ we define $\mu_{\eta^{\ssup{k}}}$ as in~\eqref{eq:Disc_WRM}
as the discrete WRM. The only difference is that feasibility is now defined as
\begin{equation*}
    \1\{\s_{\eta^\ssup{k}} \text{ is feasible}\}=\prod_{i, j\in \eta^{\ssup{k}}\colon i\sim_{g(\eta^{\ssup{k}})}j}\1\{\sigma_i\sigma_j\neq -1\},
\end{equation*}
where the constraint is that no two spins of neighboring sites $i\sim_{g(\eta^{\ssup{k}})}j$ in $g(\eta^{\ssup{k}})$ face in opposite directions. Again, $Z_{\eta^{\ssup{k}}}$ is the usual normalizing constant. In words, we consider the WRM with open boundary conditions on the finite clusters of a subcritical PGG and the measure $\K$ can then be seen as a (dependently) marked point process in $\R^d$ with mark space $\mathcal M=\{-1,0,1\}$ as illustrated in Figure~\ref{fig:PGG:clusters}. In other words, random elements $\bfeta=(\eta,\s_\eta)$ distributed according to $\K$ are locally finite marked configurations in $\R^d\times \mathcal M$, i.e., $\bfeta\in \O=\{\bfxi\subset \R^d\times \mathcal M\colon |\xi_\L|<\infty\text{ for all bounded }\L\subset\R^d\}$. In order to highlight the structural similarities in our three models, we use the same notation $\mathcal M$ and $\O$ for the mark space and the configuration space, respectively. What $\mathcal M$ and $\O$ represent will subsequently depend on the model considered in the corresponding sections.
Let us note that $\K$ is not a Potts model in the sense of~\cite{georgii1996phase} since there is in general no symmetry with respect to spin change. 
\begin{figure}
\begin{minipage}{.6\textwidth}
	\begin{center}
	   \includegraphics[width=.8\textwidth]{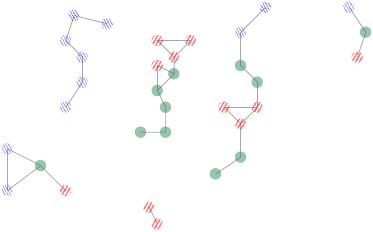}
	\end{center}	
\end{minipage}%
\begin{minipage}{.4\textwidth}
	\caption{A realization of the discrete WRM on clusters of the Poisson--Gilbert graph, represented by gray edges. In blue (resp.~green, red) the $+1$ (resp.~$0, -1$) the spins of the WRM.}
	\label{fig:PGG:clusters}
\end{minipage}
\end{figure}

Now, we reproduce the non-locality result in the more advanced continuum setting. We do this in the framework of Papangelou intensities~\cite{georgii2005conditional} as also hinted towards in the Step~7 of the proof of Theorem~\ref{thm_1} below. That is, we consider non-negative measurable functions $\rho\colon (\R^d\times \mathcal M)\times \O\to [0,\infty)$ that will serve as analogues for the specifications above and call them {\em Papangelou intensities}. 
We write $\int\d\bfx f(\bfx)=\int\d x\sum_{\s_x\in \mathcal M}f((x,\s_x))=\int\d x\sum_{\s_x\in \mathcal M}f(\bfx)$, where $\d x$ denotes the Lebesgue measure. 
\begin{definition}[GNZ equations]
    A marked point process $P$ on $\O$ {\em satisfies the GNZ equations with respect to some Papangelou intensity $\rho$} if 
    \begin{equation}\label{def:Pap}
        \int P(\d \bfeta) \sum_{x\in \eta}f(\bfx,\bfeta)=\int\d \bfx\int P(\d \bfeta)f(\bfx,\bfeta\cup\{\bfx\}) \rho(\bfx,\bfeta)
    \end{equation}
    for all measurable $f\colon (\R^d\times \mathcal M)\times\O\to [0,\infty)$. 
\end{definition}
The characterizations of Gibbs point processes with respect to Papangelou intensities is in one-to-one correspondence to the characterizations with respect to local specifications~\cite{dereudre2019introduction,flint2019functional}. More specifically, as described for example in~\cite{georgii2005conditional}, the Papangelou intensity is in one-to-one correspondence to local specifications via
\begin{align}\label{Pap-DLR}
    \g_\L(\d(\bfeta_{\ssup 1}\cdots \bfeta_{\ssup n})|\bfeta_{\L^c})=\Pi_\L(\d(\bfeta_{\ssup 1}\cdots \bfeta_{\ssup n}))
    \rho(\bfeta_{\ssup 1},\bfeta_{\L^c})\prod_{i=2}^n\rho(\bfeta_{\ssup i},\bfeta_{\ssup 1}\cdots \bfeta_{\ssup{i-1}}\bfeta_{\L^c}).
\end{align}
We are interested in locality properties of Papangelou intensities characterized as a continuity with respect to the $\tau$-topology. 
\begin{definition}[Quasilocal Papangelou intensities]
	Let $\rho$ be a Papangelou intensity. A configuration $\bfeta\in\O$ is called \emph{good} for $\rho$ if, for all  $\bfx=(x,\s_x)\in \R^d\times \mathcal M$,
	\begin{equation*}
		\sup_{\bfeta'\colon \bfeta'_\Delta=\bfeta_\Delta}\abs{\rho(\bfx,\bfeta') - \rho(\bfx,\bfeta)}\to 0\qquad \text{ as } \Delta\uparrow\R^d. 
	\end{equation*}
	We denote by $\O(\rho)$ the set of good configurations for $\rho$, while elements of $\O\setminus\O(\rho)$ are called \emph{bad} for $\rho$. We say that $\rho$ is \emph{quasilocal} if $\O(\rho)=\O$.
\end{definition}
Let us note that in~\cite{jahnel2021gibbsian} results are presented for the representability of quasilocal intensities via summable potentials. 
\begin{remark}[Topology] Let us further note that instead of the uniform $\tau$-topology, where for $\bfeta=(\eta,\s_\eta)$ we have that $\bfeta'\Rightarrow\bfeta$ if and only if $\sum_{x\in \eta'}f(\bfx)\to \sum_{x\in \eta}f(\bfx)$ for all compactly supported and bounded measurable functions $f\colon \R^d\times \mathcal M \to \R$, also the vague topology is a candidate for the definition (of absence) of convergence. In simple terms, the vague continuity would require the Papangelou intensity to be continuous with respect to perturbations for the point cloud even in bounded volumes. However, this continuity is already violated by many standard hard-core models such as the classical continuum WRM. In our case the continuity is required with respect to non-local perturbations and thus the uniform $\tau$-topology seems appropriate, see our remarks below Theorem~\ref{thm_2}.
\end{remark}

In the continuum it is necessary to distinguish quasilocality from continuity with respect to local spatial perturbations. In other words, we separate discontinuity with respect to variations of marked point clouds coming from long-range dependencies from the sensitivity to local fluctuations.
In order to accommodate this, we want to focus our attention towards a canonical class of Papangelou intensities that are graph-functions with respect to $\eta$. This means that $\rho(x^\tsup{1},\eta^\tsup{1})$ only depends on the (non-spatial) graph structure of $\eta$ and $\eta\cup\{x\}$, where $$\eta^\tsup{1}=(\eta,1_\eta)$$ 
is the all-one configuration on $\eta$. In order to make this precise, we say that $\eta'$ and $\eta$ have the {\em same graph structure} if there exists a bijection $\psi\colon \eta\to\eta'$ such that for all $x,y\in \eta$ we have that $|x-y|<1\Leftrightarrow|\psi(x)-\psi(y)|<1$. 
\begin{definition}[Graph functions]\label{def-graphpap}
	We call a mapping $\rho\colon (\R^d\times \mathcal M)\times \O\to [0,\infty)$ a {\em graph function} if 
	\begin{equation*}
		\rho(x^\tsup{1},\eta^\tsup{1})= \rho(x^\tsup{1},\eta'^{\tsup{1}})
	\end{equation*}
	whenever $\eta'$ and $\eta$, as well as  $\eta'\cup\{x\}$ and $\eta\cup\{x\}$, have the same graph structure with a bijection $\psi$ such that $\psi(x)=x$.
\end{definition}

As before, let $\eta_{\tsup{x}}\cup\{x\}$ denote the cluster of $\eta\cup\{x\}$ at $x$, where $\eta_{\tsup{x}}=(\eta_{\tsup{x}}\cup\{x\})\backslash\{x\}$ is the finite union of connected components of $\eta$ which are merged by adding $x$. Furthermore, similar as above, ${\rm n}(x,\eta)\subset\eta$ denotes the set of vertices in $\eta$ that are neighbors of $x$ in the PGG. We will always work in the setting where $\eta_{\tsup{x}}$ is finite almost-surely for any $x\in\R^d$.  

In words, the notion of a graph function allows us to consider the PGG induced by the point cloud instead of precise spatial positions. As the PGG, as a graph, is robust with respect to local perturbations, continuity questions are reduced to presence/absence of non-local spatial perturbations. As the latter are the main focus of the paper, we prefer to work with graph functions.
Let us first exhibit our canonical Papangelou intensity for $\K$. 
\begin{proposition}\label{prop_Pap_1}
$\K$ satisfies the GNZ equations with respect to the Papangelou intensity 
\begin{equation}\label{eq:PGG:pap}
\begin{split}
\rho(\bfx,\bfeta)
=\beta\frac{Z_{\eta_{\tsup{x}}}}{Z_{\eta_{\tsup{x}}\cup\{x\}}} 
\one\{\s_{{\rm n}(x,\eta)\cup\{x\}}\text{ is feasible}\}\one\{\eta_{\tsup{x}}\text{ is finite}\},
\end{split}
\end{equation}
which is also a graph function.
\end{proposition} 
As in the purely discrete model of Section~\ref{sec:DL}, the canonical Papangelou intensity~\eqref{eq:PGG:pap} (see in particular the spatially discrete analogue~\eqref{eq:Disc_Pap}) is the main ingredient for the proof of our following second main result about the absence of quasilocal graph Papangelou intensities in the strong coupling regime.
\begin{thm}\label{thm_2}
Let $\beta<\beta_{\rm c}$ and consider the symmetric WRM where $p_+=p_->0$. Then, there exists no quasilocal Papangelou intensity that is also a graph function such that the GNZ equation~\eqref{def:Pap} is satisfied for $\K$. 
\end{thm}

We highlight that, due to the additional spatial flexibility, the non-quasilocality holds on the whole symmetric parameter region. 
Before we continue, let us present some remarks on topology and soft-core WRMs.

\begin{remark}[Soft-core WRM]\label{rem_1}
As noted above, the classical continuum WRM is defined by a finite-range hard-core interaction
\begin{equation*}
    \one\{\s \text{ is feasible}\}=\prod_{i,j\in\eta\colon |i-j|<2a}\one\{\s_i\s_j\neq-1\}
\end{equation*}
and with an a-priori measure given by the superposition of two independent homogeneous PPP with intensities $\lambda_\pm$. In particular, the interaction function is local but not continuous in the vague topology. This is due to the hard-core $|i-j|<2a$. 

However, one can also consider a soft-core version of the model where $\one\{\bfs \text{ is feasible}\}$ is replaced by  
\begin{equation*}
    \psi(\s)=\prod_{i,j\in\eta}\big(1-\psi(|i-j|)\one\{\s_i\s_j=-1\}\big),
\end{equation*}
where $\psi$ is a smooth mollification function of the indicator $\one_{[0,2a)}$, e.g., some smooth decreasing function with value $1$ for all $r\in [0,2a)$ and value $0$ for $r>3a$. With this definition the soft-core WRM model is vaguely continuous. The same soft-core version can be defined for the discrete WRM on the PGG from above. Considering the proof, it will become clear that in this case, even though the soft-core WRM is vaguely continuous, the associated discrete WRM on the PGG is non-quasilocally Gibbs with respect to the vague topology.

Let us further note that in the soft-core WRM we cannot expect to find graph Papangelou intensities as distances between points may appear in which the interaction is non-trivial (non-hardcore) and varies continuously with respect to perturbing the distances. In general, the precise distances matter. This provides continuity with respect to local variations of the point cloud, as we will discuss now. 
The canonical version of the Papangelou intensity~\eqref{eq:PGG:pap} can also serve as a Papangelou intensity in the soft-core case. However, it is no-longer a graph Papangelou intensity and it is now continuous with respect to local perturbations of the point cloud. Hence, working with the vague topology is adequate for the soft-core WRM. This property (which we could call locally vague continuity) is of course a necessary requirement to be vaguely continuous. 
However, there is no vaguely continuous version of the Papangelou intensity, due to {\em non-local} reasons. 
More precisely, the canonical version of the Papangelou intensity~\eqref{eq:PGG:pap} is locally vaguely continuous, however, it is not vaguely continuous since, as in the proof of Theorem~\ref{thm_2}, we can choose a bad configuration $\eta$, as well as its $n$-dependent perturbation, based on point clouds such that the soft-core WRM interaction acts as a perfect hardcore interaction. We may choose the lattice bad configuration for this purpose, and even allow also for small enough perturbations, depending on the parameters of the mollification, as long as we stay in the perfect hardcore region. 
\end{remark}

\section{The WRM on the Poisson--Boolean model}\label{sec:PBM}
We now consider the case in which the random environment is given by a Poisson--Boolean model (PBM). 
For this, we let $\bar\eta$ denote a (iid-marked) PPP in $\R^d\times \R_+$ with intensity measure $\beta\d x\otimes \nu$, where $\beta>0$ represents the intensity of points in $\R^d$ and $\nu$ is the probability distribution of marks representing random radii. We assume that small radii have positive mass, i.e., $\nu(m<\eps)>0$ for all $\eps>0$, as well as existence of the $d$-th moment, i.e., $\int\nu(\d m)m^d<\infty$. The last property in particular implies the existence of a subcritical percolation phase for the associated PBM 
\begin{equation*}
   \BM(\bar\eta)=\bigcup_{x\in \eta}B_{m_x}(x), 
\end{equation*}
see~\cite{Gouere08}. Here, $B_{r}(x)$ denotes the open ball centered at $x$ with radius $r$. We highlight that we do not require the radii to be unbounded. 
Let us note that our modeling approach via the PBM is mathematically convenient, as a realization of the 
environment is given in terms of a marked point cloud, where radii constitute the marks. 
This may be generalized to dependent processes, allowing for dependence between either points 
or radii, or both.

Next, let us fix a constant $a>0$ and assume in the following that $\beta$ is sufficiently small such that the PBM $\BM^a(\bar\eta)=\bigcup_{x\in \eta}B_{m_x+a}(x)$ is still in the subcritical percolation regime, i.e., $\beta<\beta_{\rm c}$. Denote by $\mathcal X:=\{\bar\eta\subset \R^d\times \R_+\colon |\eta_\L|<\infty\text{ for all bounded }\L\subset\R^d\}$ the space of locally finite point clouds with marks in $\R_+$.

Now, let us present the continuum WRM, which will then be put into the random environment given by the PBM. For any bounded Borel set $A\subset\R^d$, we consider the (finite-volume) continuum WRM
\begin{equation*}
    \mu_A(\d\bar\s)=\frac{1}{Z_A}\Pi^\pm_A(\d \bar\s)\one\{\bar\s \text{ is feasible}\},
\end{equation*}
where $\Pi^\pm_A(\d \bar\s)$ is short-hand notation for the product of two independent PPPs $\Pi^+_A(\d \s^+)$ and $\Pi^-_A(\d \s^-)$ on $\R^d$ with intensity measures $\lambda_+\1\{x\in A\}\d x$ and $\lambda_-\1\{x\in A\}\d x$, respectively, and which represent particles with color $+$ and $-$, respectively. Feasibility is defined as the property that, for all $x\in \s^+$ and $y\in \s^-$, we have that $B_a(x)\cap B_a(y)=\emptyset$. Let us note that, due to the Poisson thinning theorem, we can equivalently consider $\Pi^\pm_A$ to be a single PPP with intensity $(\lambda_++\lambda_-)\one\{x\in A\}\d x$, which is iid marked in $\mathbb M=\{+,-\}$ with a Bernoulli mark probability $\lambda_+/(\lambda_++\lambda_-)$ to see a $+$ mark. Consequently, we can think of $\bar\s=(\s^+,\s^-)$ as a two-color configuration, i.e., a subset of $\R^d\times\mathbb M$.
With a further bijection, we can then also identify it with the pair $(\o,\s_\o)$, where $\o=\s^+\cup\s^-\subset\R^d$ is the gray configuration of all particles and $\s_\o\in \{+,-\}^\o$ is the coloring of the gray configuration $\o$. Again, $Z_A$ is the associated normalization function that turns $\mu_A$ into a probability measure on $\mathcal X^\pm:=\{\bar\s\subset\R^d\times \mathbb M\colon |\o_\L|<\infty\text{ for all bounded }\L\subset\R^d\}$, the space of locally finite point clouds with marks in $\mathbb M$. For a WR configuration $\bar\s\in\mathcal X^\pm$ and a fixed set $A\subset\R^d$, we denote by $\bar\s_A$ the restriction of the configuration to the set $A\times\mathbb M$, that is $\bar\s_A = \bar\s\cap (A\times\mathbb M)$.

\medskip
We are interested in locality properties of the joint probability measure
\begin{equation}\label{eq:PBM:joint}
\mathbb K(\d\bar\eta,\d\bar\s):=\Pi(\d\bar\eta)\mu_{\BM(\bar\eta)}(\d\bar\s),
\end{equation}
where $\mu_{\BM(\bar\eta)}$ should be understood as the (countable) infinite product of WRM measures for each of the (almost-surely) finite connected regions of $\BM^a(\bar\eta)$.
We highlight that the assumption of subcriticality of $\BM^a(\bar\eta)$ ensures that indeed WRMs in different clusters do not interact. Let us note that the marginal $\int\mathbb K(\d\bar\eta,\d\bar\s)f(\bar\s)$, or annealed measure, can be seen as a Cox--Gibbs point process since the underlying PPP is governed by a random intensity measure based on $\bar\eta$, see also~\cite{jahnel2025phase}.

\medskip
For our analysis, it is essential to relate the joint model $\mathbb K$ to a dependently marked PPP. For this, in analogy to the model in Section~\ref{sec:PGG}, we first equip each point $x\in\eta$, in the PPP of the environment, independently with an extended mark $(m_x,\bar\s_\ssup{x})\in\mathcal M:=\R_+\times\mathcal X^\pm$ consisting of both its radius $m_x$ and a two-layer Poisson point cloud $\bar\s_\ssup{x} = \bar\s+x$, governed by the joint probability measure $\nu(\d m)\Pi^\pm_{B_m(o)}(\d\bar\s)$, where $\bar\s+x=(\o+x,\s_{\o+x})$.
This iid-marked PPP serves as an a-priori measure on the space of locally finite point configurations in $\R^d$ with marks in $\mathcal M$, i.e., on $\O:=\{\bfeta\subset\R^d\times \mathcal M\colon |\eta_\L|<\infty\text{ for all bounded }\L\subset\R^d\}$. This implies that we work with integrations of the form
$$\int\d \bfx f(\bfx) = \int \d x \int \nu(\d m)\int\Pi_{B_m(o)}^{\pm}(\d\bar\s)f(x, (m_x,\bar\s+x)).$$
We further consider non-negative measurable Papangelou intensities $\rho\colon (\R^d\times \mathcal M)\times \O\to [0,\infty)$ that give rise to infinite-volume Gibbs point processes via the GNZ equations. 
\begin{definition}[GNZ equations]
A marked point process $P$ on $\O$ {\em satisfies the GNZ equations with respect to the Papangelou intensity $\rho$}, if
\begin{align}\label{eq:PBM:GNZ_2}
  \int P(\d\bfeta)\sum_{\bfx\in\bfeta}f(\bfx,\bfeta)=\int\d\bfx\int P(\d\bfeta)f(\bfx,\bfeta\cup\{\bfx\}) \rho(\bfx,\bfeta),
\end{align}
 for all $f\colon (\R^d\times\mathcal M)\times \O\to [0,\infty)$.
\end{definition}

Now, our joint measure $\mathbb K$ can be understood as a projection of a marked point process on $\O$, based on the iid-marked PPP described above, in the following sense. 
\begin{lemma}\label{lem:Model3}
There exists a marked point process $\mathbb K^\ssup{1}$ on $\O$ such that,
for all measurable functions $f\colon\mathcal X\times \mathcal X^\pm\to[0,\infty)$, we have that 
$$
\int\mathbb K(\d\bar\eta,\d\bar\s)f(\bar\eta,\bar\s)=\int\mathbb K^\ssup{1}(\d\bfeta)f(\bar\eta, \bigcup_{x\in \eta}\bar\s_\ssup{x}). 
$$
\end{lemma}
The process $\mathbb K^\ssup{1}$ is generated by a random algorithm that assigns WR-points in overlapping environment balls as marks to points $x\in\eta$. 
 The construction of $\mathbb K^\ssup{1}$ is given in the Appendix in Section~\ref{sec:ap}. 
 Although one can use other constructions, 
 the correct treatment of overlaps is essential, since adding points may also create overlaps. We establish in Proposition~\ref{prop_3} that the measure $\mathbb K^\ssup{1}$ satisfies the GNZ equations~\eqref{eq:PBM:GNZ_2} with respect to a Papangelou intensity $\rho$.

\medskip
If we focus on properties of the joint measure $\mathbb K$ on $\mathcal X\times\mathcal X^\pm$ itself, a more direct approach is available: Although $\mathbb K$ is not a marked point process, we can show that it satisfies a generalized form of the GNZ equations, see~\eqref{eq:PBM:GNZ_3} below. 
Here a generalized intensity function appears, which is a function  $\rho\colon((\R^d\times\R_+)\times\mathcal X^\pm)\times(\mathcal X\times\mathcal X^\pm)\to [0,\infty)$. We call this intensity function  again simply  a Papangelou intensity. It allows the addition of a point $(\bar x,\bar\sigma_{\BM(\bar x)})$ to a configuration $(\bar\eta,\bar\sigma)$ only if the overlapping region is free of points coming from $\bar\sigma$, 
see Figure~\ref{fig:PBM:pap}.
Like in the previous sections, we are interested in locality properties of this Papangelou intensity, in the following sense.
\begin{figure}
\begin{minipage}{.45\textwidth}
    \begin{center}
	\includegraphics[width=.9\textwidth]{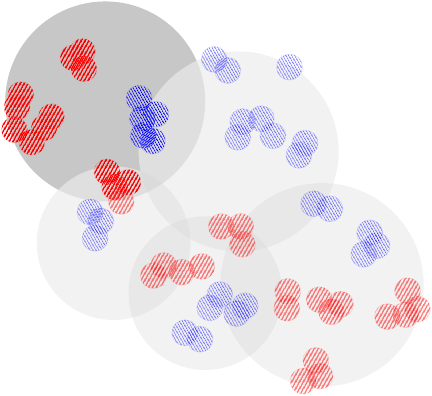}
    \end{center}	
\end{minipage}\hfill
\begin{minipage}{.5\textwidth}
    \caption{A realization of a WRM on the clusters of a PBM, represented in gray. In blue (resp.~red) the $+1$ (resp.~$-1$) marks of the WRM.
    According to \eqref{eq:PBM:pap}, adding the point $(\bar x,\bar\s_{\BM(\bar x)})$ given by the dark-gray environment point and the highlighted WR points is possible only if there are no WR points of $(\bar\eta,\bar\sigma)$ in the overlap $\BM(\bar x)\cap\BM(\bar\eta)$.}
    \label{fig:PBM:pap}
\end{minipage}    
\end{figure}

\begin{definition}[Quasilocality]
A configuration $(\bar\eta,\bar\sigma)\in \mathcal X\times\mathcal X^\pm$ is called \emph{good} for $\rho\colon((\R^d\times\R_+)\times\mathcal X^\pm)\times(\mathcal X\times\mathcal X^\pm)\to [0,\infty)$ if, for all $(\bar x,\bar\sigma_{(x)})\in (\R^d\times\R_+)\times\mathcal X^\pm$,
	\begin{equation*}
		\sup_{(\bar\eta',\bar\sigma')\colon (\bar\eta'_\Delta,\bar\sigma'_\Delta)=(\bar\eta_\Delta,\bar\sigma_\Delta)}\abs{\rho((\bar x,\bar\sigma_{(x)}),(\bar\eta',\bar\sigma')) - \rho((\bar x,\bar\sigma_{(x)}),(\bar\eta,\bar\sigma))}\to 0\qquad \text{ as } \Delta\uparrow\R^d. 
	\end{equation*} 
    We say that $\rho$ is \emph{quasilocal} if all configurations in $\mathcal X\times\mathcal X^\pm$ are good for $\rho$.
\end{definition}

As a replacement for the notion of graph functions of our second model, we consider the following robustness with respect to local perturbations in the environment. This is done, as mentioned above in the context of graph functions for the PGG, to remove potential additional discontinuities in Papangelou intensities that only come from a local sensitivity. More precisely, we define $\varepsilon$-balls around finite subsets $\bar\eta^{\ssup f}=\{(y_1,m_{y_1}),\dots, (y_l,m_{y_l})\}\subset\bar\eta$ via
\begin{align*}
B_\varepsilon(\bar\eta^{\ssup f},\bar\eta)=\{\bar\zeta=\bar\zeta^{\ssup f}&\cup(\bar\eta\setminus \bar\eta^{\ssup f})\in \O\colon \bar\zeta^{\ssup f}=\{(z_1,m_{z_1}),\dots, (z_l,m_{z_l})\}\\
&\text{ and for all }i\in\{1,\dots,l\}, \, \zeta^{\ssup f}\cap B_\varepsilon(y_i)=z_i, |m_{z_i}-m_{y_i}|<\varepsilon\}.
\end{align*}
In words, $B_\varepsilon(\bar\eta^{\ssup f},\bar\eta)$ is the set of configurations which are perturbations, both with respect to the position and the radius mark, of a finite subset of a given environment $\bar\eta$. Similar to the previous model, for some $\bar y\in \R^d\times\R_+$, we denote by $\bar\eta_\tsup{\bar y}$ the set of marked Poisson points in $\bar\eta\setminus \{\bar y\}$ that give rise to the cluster in $\BM^a(\bar\eta\cup\{\bar y\})$ that contains $y$. 
In this setting, let $\bar\eta^\tsup{1}=(\eta,(m_\eta,\emptyset))\in \mathcal X\times\mathcal X^\pm$ be the configuration where the WR-part of the mark consists of the empty configuration.
\begin{definition}[Local robustness]\label{def:PBM-locrob}
    We say that a function $\rho\colon((\R^d\times\R_+)\times\mathcal X^\pm)\times(\mathcal X\times\mathcal X^\pm)\to [0,\infty)$ is {\em locally robust} if, for all $\bar x\in\R^d\times\R_+$ and all $\bar\eta\in \mathcal X$ such that $\BM^a(\bar\eta)$ contains no infinite component, we have that 
    \begin{equation*}
        \sup_{\bar\eta'\in B_\varepsilon(\bar\eta_\tsup{\bar x},\bar\eta)}\abs{\rho(\bar x^\tsup{1},\bar\eta'^{\tsup{1}}) - \rho(\bar x^\tsup{1},\bar\eta^\tsup{1})}\to 0\qquad \text{ as } \varepsilon\to0. 
    \end{equation*}
\end{definition}

Correspondingly to the previous models, we write ${\rm n}(\bar x,\bar\eta)\subset\bar\eta$ for the vertices in $\bar y\in\bar\eta$ that such that $\BM(\bar y)\cap\BM(\bar x)\neq\emptyset$.
\begin{proposition}\label{prop_2}
The measure $\mathbb K$ on $\mathcal X\times\mathcal X^\pm$ satisfies
\begin{equation}\label{eq:PBM:GNZ_3}
\begin{split}
  &\int \mathbb K(\d(\bar\eta,\bar\sigma))\sum_{\bar x\in\bar\eta}f((\bar x,\bar\sigma_{\BM(\bar x)}),(\bar\eta,\bar\sigma))=\int\d\bar x\int \Pi^\pm_{\BM(\bar x)}(\d \bar\sigma')\\
  &\qquad\qquad 
  \int \mathbb K(\d(\bar\eta,\bar\sigma))f\big((\bar x,\bar\s'),(\bar\eta\cup\{\bar x\},\bar\s_{\BM(\bar\eta)}\bar\s'_{\BM(\bar x)})\big) \rho\big((\bar x,\bar\s'),(\bar\eta,\bar\s)\big),
\end{split}
\end{equation}
for all $f\colon ((\R^d\times\R_+)\times \mathcal X^\pm)\times (\mathcal X\times\mathcal X^\pm) \to [0,\infty)$,
with respect to the locally robust Papangelou intensity
\begin{equation}\label{eq:PBM:pap}
\begin{split}
  \rho\big((\bar x,\bar\s'),(\bar\eta,\bar\s)\big) = &\beta\frac{Z_{\BM(\bar\eta_\tsup{\bar x})}}{Z_{\BM(\bar\eta_\tsup{\bar x}\cup\bar x)}}\one\{\bar\s'_{\BM(\bar x)}\bar\s_{\BM({\rm n}(\bar x,\bar \eta))\setminus \BM(\bar x)} \text{ is feasible}\} \chi_{\bar x,\bar\eta_\tsup{\bar x}}(\bar\s),
\end{split}
\end{equation}
where $\chi_{\bar x,\bar\eta_\tsup{\bar x}}(\bar\s) = \e^{(\lambda_++\lambda_-)\abs{\BM(\bar x)\cap \BM(\bar\eta_\tsup{\bar x})}}\one\{\bar\s_{\BM(\bar x)\cap\BM(\bar\eta_\tsup{\bar x})}=\emptyset\}\one\{\bar\eta_{\tsup{\bar x}}\text{ is finite}\}$.
\end{proposition}
Note that, compared to the corresponding canonical Papangelou intensities~\eqref{eq:Disc_Pap} and~\eqref{eq:PGG:pap}, in~\eqref{eq:PBM:pap} we find an additional term $\chi_{\bar y,\bar\eta_\tsup{\bar y}}$ that removes overlapping points from the boundary condition.
Finally, we present our third main result that states that, for large enough symmetric coupling, there is no locally robust quasilocal Papangelou intensity for $\mathbb{K}$ but, as in the lattice case, for uniformly bounded radii, i.e., $\nu(m<T)=1$ for some $T>0$, and sufficiently weak coupling, the quasilocality is restored. 

\begin{thm}\label{thm:PBM:nG}
Let $\beta<\beta_{\rm c}$ and consider the symmetric WRM, where $\lambda_+=\lambda_-$ is sufficiently large. Then, there exists no locally robust quasilocal Papangelou intensity such that~\eqref{eq:PBM:GNZ_3} is satisfied for $\mathbb K$. On the other hand, if the radii are uniformly bounded, we have that for all sufficiently small $\lambda_+=\lambda_-$, the canonical Papangelou intensity~\eqref{eq:PBM:pap} is quasilocal. 
\end{thm}
We highlight that the phase transition of existence and absence of quasilocality is similar in the fully continuous and in the lattice case of Section~\ref{sec:DL}, but it is different in the intermediate model of Section~\ref{sec:PGG}. This is a consequence of the fact that, from the perspective of the WRM, the intermediate model has much less spatial structure compared to the other two models.   
Let us finally mention that the requirement that $\nu(m<\eps)>0$ for all $\eps>0$ is for convenience of the proof only. The results should also hold for general (even deterministic) marks that guarantee existence of a subcritical percolation phase. Let us also mention that the requirement of subcriticality of the environment is mostly for convenience since then we do not have to define infinite-volume WRMs on general domains. We expect the results to hold also for $\beta\ge \beta_{\rm c}$. Let us finally mention that, while it is true that in general quasilocality is not as frequently occurring for continuum models as it is for discrete models, the surprising phenomenon here is that, even in a finite-range setting, where quasilocality would still be expected to hold, the presence of an environment breaks this continuity. The relevance of quasilocality for continuum specifications is further highlighted by the possibility of representing them through hyperedge potentials, very much in the spirit of the known Kozlov-type theorems on the lattice, see~\cite{jahnel2021gibbsian}.

\section{Proofs}\label{sec_proofs}
\begin{proof}[Proof of Theorem~\ref{thm_1}]
Let us start with the proof of non-quasilocal Gibbsianness. We denote $p:=p_+=p_-$. 

\medskip
{\bf Step 1:} 
As a first step, we consider the canonical candidate for a specification given by the conditional measure at the origin $o\in \Z^d$ given some typical boundary condition. To be more precise, let $o\in \eta$ be a gray configuration such that the maximally connected component at the origin $o$, the {\em cluster} of occupied sites containing the origin, is finite. Furthermore,  consider a boundary condition $\eta^\tsup{1}_{o^c}=(\eta\setminus\{o\},1_{\eta\setminus\{o\}})$, where again $1_\eta$ denotes the all-one coloring of the gray configuration $\eta$ and $o^{\rm c}=\Z^d\setminus\{o\}$. Then, using the independence of the WRM with respect to disjoint clusters in $\eta$, we have that
\begin{equation}\label{eq:DL:JointCond}
\begin{split}
 	K(\s_o=1\vert \eta^\tsup{1}_{o^c})=\frac{q\, p}{q(1-p)+(1-q)Z_{\eta_\tsup{o}\cup \{o\}}/Z_{\eta_\tsup{o}}}. 
\end{split}
\end{equation} 
Note that all terms in~\eqref{eq:DL:JointCond} except for $Z_{\eta_\tsup{o}\cup\{o\}}/Z_{\eta_\tsup{o}}$ are of course strictly spatially local with respect to $\eta_{o^c}$. They are even Markovian in the sense that only the configuration in the one-step neighborhood of the origin is relevant.  

\medskip
{\bf Step 2:}
In this step, we consider a particular choice of boundary condition $\eta_{o^c}$ in order to exhibit a discontinuity. More precisely, consider the half spaces $\L_+=\{(x_1,\dots, x_d)\in \Z^d\colon x_1>0\}$ and $\L_-=\{(x_1,\dots, x_d)\in \Z^d\colon x_1<0\}$ and let $\L^n_+=\L_+\cap (Q_n(o)\cup Q_n(2n e_2))$ and $\L^n_-=\L_-\cap (Q_n(o)\cup Q_n(2n e_2))$, where $Q_n(x)=[-n,n]^d+x$ is the $n$-box centered at $x\in \Z^d$ and $e_i=(0_1,\dots, 0_{i-1},1_i,0_{i+1},\dots, 0_d)$ is the unit vector in the $i$-th coordinate. In words, $\L^n_\pm$ are two disjoint rectangles which are symmetric towards each other and also symmetric with respect to the reflection that identifies $o$ and $2ne_2$, see Figure~\ref{fig:lattice_pt}.
\begin{figure}[ht]
\begin{minipage}{.4\textwidth}
	\begin{center}
		\includegraphics[width=.7\textwidth]{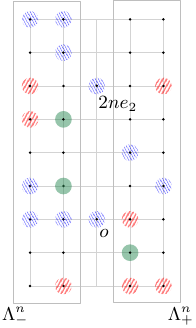}
	\end{center}	
\end{minipage}\hfill
\begin{minipage}{.6\textwidth}
	\caption{A realization of a WRM on the truncated half-lattices $\L_\pm^{n}$. Black vertices represent occupation in the underlying Bernoulli site percolation. Blue (resp.~green, red) disks represent the $+$ (resp.~$0$, $-$) spins in the WRM.}
	\label{fig:lattice_pt}
\end{minipage}
\end{figure}
There are many other possible choices for such regions, however, this choice seems computationally convenient.

Now let $\L^n=\L^n_-\cup\L^n_+$, $\zeta^n=\L^n\cup\{o\}$ and note that it is only the origin that creates a connection between the symmetric sets since $\zeta^n\setminus\{o\}=\L^n$. We want to compare 
$Z_{\zeta^n}/Z_{\zeta^n\setminus\{o\}}$ with $Z_{\zeta'^n}/Z_{\zeta'^n\setminus\{o\}}$ where $\zeta'^n=\zeta^n\cup \{2ne_2\}$. Note that in $\zeta'^n$ there is an additional connection between $\L^n_-$ and $\L^n_+$ at the point $2ne_2$ present for any $n$, however, both $\zeta^n$ and $\zeta'^n$ share the same limiting configuration $\zeta=\L_-\cup \L_+\cup\{o\}=\L\cup\{o\}$, as $n$ tends to infinity. Hence, in order to show that any boundary condition $\eta_{o^c}$ with $\zeta_{o^c}=\L$ is a discontinuity point of $\eta_{o^c}\mapsto K(\s_o=1\vert \eta^\tsup{1}_{o^c})$, it suffices to show that 
\begin{equation*}
\begin{split}
\liminf_{n\uparrow\infty}\Big|\frac{Z_{\zeta'^n}}{Z_{\zeta'^n\setminus\{o\}}}-\frac{Z_{\zeta^n}}{Z_{\zeta^n\setminus\{o\}}}\Big|>0.
\end{split}
\end{equation*}
For this, note that for all $n\ge 1$, with $z_n=2ne_2$, we have $Z_{\zeta^n\setminus\{o\}}/Z_{\zeta'^n\setminus\{o\}}=Z_{\L^n}/Z_{\L^n\cup\{z_n\}}\ge 1$, since $\1\{\s_{\L^n\cup\{z_n\}}\text{ is feasible}\}\le \1\{\s_{\L^n}\text{ is feasible}\}$, and thus, 
\begin{equation*}
\begin{split}
\Big|\frac{Z_{\zeta'^n}}{Z_{\zeta'^n\setminus\{o\}}}-\frac{Z_{\zeta^n}}{Z_{\zeta^n\setminus\{o\}}}\Big|=\frac{Z_{\zeta^n\setminus\{o\}}}{Z_{\zeta'^n\setminus\{o\}}}\Big|\frac{Z_{\zeta'^n}}{Z_{\zeta^n\setminus\{o\}}}-\frac{Z_{\zeta^n}Z_{\zeta'^n\setminus\{o\}}}{Z_{\eta^n\setminus\{o\}}Z_{\zeta^n\setminus\{o\}}}\Big|\ge \Big|\frac{Z_{\zeta'^n}}{Z_{\zeta^n\setminus\{o\}}}-\frac{Z_{\zeta^n}Z_{\zeta'^n\setminus\{o\}}}{Z_{\zeta^n\setminus\{o\}}Z_{\zeta^n\setminus\{o\}}}\Big|.
\end{split}
\end{equation*}
Hence, using the symmetry with respect to $o$ and $z_n$, it suffices to show that 
\begin{equation}\label{eq:DL:JointCond_2}
\begin{split}
\liminf_{n\uparrow\infty}\Big|\frac{Z_{\L^n\cup\{o,z_n\}}}{Z_{\L^n}}-\Big(\frac{Z_{\L^n\cup\{o\}}}{Z_{\L^n}}\Big)^2\Big|>0.
\end{split}
\end{equation}

\medskip
{\bf Step 3:}
Now, in order to prove~\eqref{eq:DL:JointCond_2}, let us denote by $x_o$, respectively $y_o$, the unique neighbor of $o$ in $\L^n_-$, respectively $\L^n_+$. Then, making explicit the probability with respect to the origin, we have that
\begin{equation*}
	\frac{Z_{\L^n\cup\{o\}}}{Z_{\L^n}}=\mu_{\L^n}(\underbrace{1 - p + p\1\{\sigma_{x_o}=\sigma_{y_o}=0\} - p\1\{\sigma_{x_o}\sigma_{y_o}=1\}}_{=:\phi(\sigma_{x_o},\sigma_{y_o})}),
\end{equation*}
and analogously for $Z_{\L^n\cup\{z_n\}}/Z_{\L^n}$ with $x_n$, respectively $y_n$, the unique neighbor of $z_n$ in $\L^n_-$, respectively in $\L^n_+$. Similarly, making explicit the probability with respect to $o$ and $z_n$, we have 
\begin{equation*}
	\frac{Z_{\L^n\cup\{o, z_n\}}}{Z_{\L^n}}=\mu_{\L^n}(\phi(\sigma_{x_o},\sigma_{y_o})\phi(\sigma_{x_n},\sigma_{y_n})),
\end{equation*}
and thus~\eqref{eq:DL:JointCond_2} can be seen as a statement for the covariance of $\phi(\sigma_{x_o},\sigma_{y_o})$ and $\phi(\sigma_{x_n},\sigma_{y_n})$ under the measure $\mu_{\L^n}$. Since constants do not play a role in the covariance, it suffices to prove that for all $p$ sufficiently close to $1/2$,
\begin{equation}\label{eq:DL:JointCond_4}
\begin{split}
\liminf_{n\uparrow\infty}\Cov_n\big(\1\{\sigma_{x_o}\hspace{-0.15cm}=\sigma_{y_o}\hspace{-0.15cm}=0\} - \1\{\sigma_{x_o}\sigma_{y_o}\hspace{-0.15cm}=1\}; \1\{\sigma_{x_n}\hspace{-0.15cm}=\sigma_{y_n}\hspace{-0.15cm}=0\} - \1\{\sigma_{x_n}\sigma_{y_n}\hspace{-0.15cm}=1\}\big)>0,
\end{split}
\end{equation}
where $\Cov_n$ denotes the covariance with respect to $\mu_{\L^n}$. We first focus on the dominant term $\Cov_n\big(\1\{\sigma_{x_o}\sigma_{y_o}=1\};\1\{\sigma_{x_n}\sigma_{y_n}=1\}\big)$.

\medskip
{\bf Step 4:} We claim that, for all $p$ sufficiently close to $1/2$, we have that,
\begin{equation*}
\begin{split}
\liminf_{n\uparrow\infty}\Cov_n\big(\1\{\sigma_{x_o}\sigma_{y_o}=1\};\1\{\sigma_{x_n}\sigma_{y_n}=1\}\big)>1/5. 
\end{split}
\end{equation*}
Indeed, first note that by independence and symmetry with respect to the halfspaces $\L_\pm^n$,
\begin{equation}\label{eq:DL:JointCond_3}
\begin{split}
\Cov_n\big(\1\{\s_{x_o}\s_{y_o}=1\};\1\{\s_{x_n}\s_{y_n}=1\}\big)=&\ 2\big(\mu_{\L^n_-}(\s_{x_o}=\s_{x_n}=1)^2\\
&+\mu_{\L_-^n}(\sigma_{x_o}=-\sigma_{x_n}=1)^2- 2\mu_{\L_-^n}(\sigma_{x_o}=1)^4\big),
\end{split}
\end{equation}
where 
\begin{equation*}
\begin{split}
0\le 1/2-\mu_{\L_-^n}(\sigma_{x_o}=1)=\mu_{\L_-^n}(\sigma_{x_o}=0)/2.
\end{split}
\end{equation*}
We will show in the next step below that $\mu_{\L_-^n}(\sigma_{x_o}=0)$ tends to zero as $p$ tends to $1/2$. 
In order to estimate the other two terms in~\eqref{eq:DL:JointCond_3}, we employ the random-cluster representation $R_{\L_-^n}$ of $\mu_{\L_-^n}$ see~\cite[Lemma 5.1]{higuchi_2004}. In words, $R_{\L_-^n}$ is a measure on  $\{0,1\}^{\L_-^n}$ in which sites interact via the number of (gray) clusters. In particular, since $\mu_{\L^n_-}(\s_{x_o}=\s_{x_n}=1)\le 1/2$, 
\begin{equation*}
\begin{split}
0\le 1/2-\mu_{\L^n_-}(\s_{x_o}=\s_{x_n}=1)\le R_{\L^n_-}(x_o\not\leftrightsquigarrow x_n)/2,
\end{split}
\end{equation*}
where $\{x_o\not\leftrightsquigarrow x_n\}$ denotes the event that $x_o$ and $x_n$ do not belong to the same cluster, which is a decreasing event with respect to the natural partial ordering on $\{0,1\}^{\Z^d}$. In particular, note that the random-cluster measure $R_{\L_-^n}$ is bounded from below by a Bernoulli field $P$ on $\{0,1\}^{\Z^d}$, via Holley's inequality, see~\cite[Lemma 5.4]{higuchi_2004}, with a parameter that tends to one as $p$ tends to $1/2$. Hence, $R_{\L^n_-}(x_o\not\leftrightsquigarrow x_n)\le P(x_o\not\leftrightsquigarrow x_n)$ and $P(x_o\not\leftrightsquigarrow x_n)$ tends to $1-\theta^2$ as $n$ tends to infinity, where $\theta$ denotes the percolation probability. However, $\theta$ tends to one as $p$ tends to $1/2$. In a similar fashion, 
\begin{equation*}
\begin{split}
0\le\mu_{\L^n_-}(\s_{x_o}=-\s_{x_n}=1)\le R_{\L^n_-}(x_o\not\leftrightsquigarrow x_n)/2,
\end{split}
\end{equation*}
where the right-hand side tends to $1-\theta^2$ as $n$ tends to infinity and $\theta$ tends to one as $p$ tends to $1/2$. Using this, we arrive at 
\begin{equation*}
\begin{split}
\liminf_{n\uparrow \infty}\Cov_n\big(\1\{\s_{x_o}\s_{y_o}=1\};\1\{\s_{x_n}\s_{y_n}=1\}\big)\ge 1/4-\eps(p),
\end{split}
\end{equation*}
for some function $p\mapsto \eps(p)$ with $\lim_{p\uparrow 1/2}\eps(p)=0$, as desired. 

\medskip
{\bf Step 5:} Now, in order to finish the proof of the statement~\eqref{eq:DL:JointCond_4}, note that by bounding indicator functions by one, all the remaining terms can be bounded by $\mu_{\L_-^n}(\sigma_{x_o}=0)$. 
In order to show that $\mu_{\L_-^n}(\sigma_{x_o}=0)$ tends to zero as $p$ tends to $1/2$, we again invoke the random-cluster representation $R_{\L_-^n}$ and the dominating Bernoulli field $P$ and note that 
\begin{equation*}
\begin{split}
\mu_{\L_-^n}(\sigma_{x_o}=0)=R_{\L_-^n}(\o_{x_o}=0)\le P(\o_{x_o}=0)
\end{split}
\end{equation*}
since $\{\o_{x_o}=0\}$ is a decreasing event. However, $P(\o_{x_o}=0)$ tends to zero as $p$ tends to $1/2$ and this finishes the proof that $\eta_{o^c}\mapsto K(\s_o=1\vert \eta^\tsup{1}_{o^c})$ is not continuous at $\zeta^\tsup{1}$ with $\zeta=\L\cup\{o\}$. 

\medskip
{\bf Step 6:} As a final step in this part of the proof, we need to lift the discontinuity to all specifications $\g$ that satisfy the DLR equations with respect to $K$. For this, it suffices to prove that $\zeta^\tsup{1}$ with $\zeta=\L\cup\{o\}$ is also a point of discontinuity of $\bfeta\mapsto \g_o(\s_o=1\vert \bfeta_{o^c})$. Let us assume the contrary, namely that
\begin{equation*}
\begin{split}
\lim_{\Delta\uparrow\Z^d}\sup_{\bfeta\colon \bfeta_\Delta=\zeta^\tsup{1}_\Delta}\abs{\g_o(\s_o=1\vert \bfeta_{o^c})-\g_o(\s_o=1\vert \zeta^\tsup{1}_{o^c})}=0.
\end{split}
\end{equation*}
Then, by a version of the Lebesgue differentiation lemma, see for example~\cite[Lemma 3.6]{JaKo22}, we have that 
\begin{equation*}
\begin{split}
\lim_{n\uparrow \infty} \frac{1}{K(A_n)}\int_{A_n}K(\d \bfeta)\g_o(\s_o=1\vert \bfeta) =\g_o(\s_o=1\vert \zeta^\tsup{1}_{o^c}), 
\end{split}
\end{equation*}
where 
\begin{equation*}
    \O\supset A_n=\{(\eta,\s_\eta)\colon \s_{\zeta^n}=1\text{ and } \eta_i=u\text{ for all }i\in\partial\zeta^n\}
\end{equation*}
is the set of configurations that have a cluster $\zeta^n$ of color one that is surrounded by unoccupied sites and that are otherwise unspecified. Note that $A_n$ is a Borel set and $K(A_n)>0$ for all $n\ge 1$.

However, due to the DLR equations and the locality of canonical conditional expectation from Step 1, 
\begin{equation*}
\begin{split}
\frac{1}{K(A_n)}\int_{A_n}K(\d \bfeta)\g_o(\s_o=1\vert \bfeta)=\frac{1}{K(A_n)}\int_{A_n}K(\d \bfeta)K(\s_o=1\vert \bfeta)=K(\s_o=1\vert \zeta_{o^c}^{n,\tsup{1}})
\end{split}
\end{equation*}
and the same holds for $\zeta^{n,\tsup{1}}$ replaced by $\zeta'^{n,\tsup{1}}$. Hence,                 
\begin{equation*}
\begin{split}
\lim_{n\uparrow \infty}\big(\g_o(\s_o=1\vert \zeta^{n,\tsup{1}}_{o^c})-K(\s_o=1\vert \zeta_{o^c}^{n,\tsup{1}})\big)=0=\lim_{n\uparrow \infty}\big(\g_o(\s_o=1\vert \zeta'^{n,\tsup{1}}_{o^c})-K(\s_o=1\vert \zeta'^{n,\tsup{1}}_{o^c})\big). 
\end{split}
\end{equation*}
But by the previous steps, $\liminf_{n\uparrow \infty}\big|K(\s_o=1\vert \zeta'^{n,\tsup{1}}_{o^c})-K(\s_o=1\vert \zeta_{o^c}'^{n,\tsup{1}})\big|>0$
and thus also 
$\liminf_{n\uparrow \infty}\big|\g_o(\s_o=1\vert \zeta^{n,\tsup{1}}_{o^c})-\g_o(\s_o=1\vert \zeta'^{n,\tsup{1}}_{o^c}))\big|>0$,
a contradiction. This finishes this part of the proof.

\medskip
{\bf Step 7:} It remains to prove the quasilocal Gibbsianness for all sufficiently small $p_+,p_-$. For this, note that it suffices to check the continuity for the discrete Papangelou intensity $\bfeta\mapsto \rho_o(\bfeta_o,\bfeta_{o^{\rm c}})$ as presented in~\eqref{eq:Disc_Pap} and repeated here for convenience,
\begin{equation*}
\begin{split}
\bfeta\mapsto \rho_o(\bfeta_o,\bfeta_{o^{\rm c}})=\frac{K(\bfeta_o|\bfeta_{o^{\rm c}})}{K(u_o|\bfeta_{o^{\rm c}})}
=\Biggl(\frac{q}{1-q}\prod_{i\in\eta\colon \abs{i}=1}\1\{\sigma_i\sigma_o\neq -1\}\Biggr)
\frac{Z_{\eta_\tsup{o}}}{Z_{\eta_\tsup{o}\cup\{o\}}}
\end{split}
\end{equation*}
with $\bfeta_o\in \{-1,0,1\}$ and already assuming that the boundary condition is feasible. Indeed, we can factorize $K(\bfeta_\L|\bfeta_{\L^{\rm c}})$ into finitely many single-site terms, use that $K(u_o|\bfeta_{o^{\rm c}})$ is uniformly bounded away from zero and translation invariance. Now, in order to see the continuity at small $p_+,p_-$, we note that the inverse fraction 
of partition functions can be rewritten as the WR expectation in the volume 
$\eta_\tsup{o}$ 
of the function $h(\bfeta_{\partial o})=\sum_{\s_o=+1,-1,0}p_{\s_o}\prod_{i\in\eta\colon \abs{i}=1}\1\{\sigma_i\sigma_o\neq -1\}$ depending on the spins at the boundary of the origin, i.e., 
\begin{equation*}
\begin{split}
&\frac{Z_{\eta_\tsup{o}\cup\{o\}}}{Z_{\eta_\tsup{o}}}
=\mu_{\eta_\tsup{o}}\big(h(\bfeta_{\partial o})\big).
\end{split}
\end{equation*}
The dependence of this quantity on the change of the volume $\eta_\tsup{o}$ 
can now be estimated by standard uniqueness techniques like the Dobrushin criterion. In particular, for all sufficiently small $p_+,p_-$ the WRM is in the uniqueness regime, which implies continuity of $\eta\mapsto \mu_{\eta_\tsup{o}}\big(h(\bfeta_{\partial o})\big)$. 
\end{proof}

\begin{proof}[Proof of Proposition~\ref{prop_Pap_1}]
First note that $\rho$ is clearly a graph function as defined above. We claim that this is also a Papangelou intensity for $\K$.
Indeed, by the Mecke theorem for PPPs,
\begin{equation*}
\begin{split}
\int&\d\bfx\int \K(\d \bfeta)f(\bfx,\bfeta\cup\{\bfx\}) \rho(\bfx,\bfeta)\\
  &=\int\d\bfx\int \Pi(\d \eta)\sum_{\s_\eta}\mu_\eta(\s_\eta)f\big((x,\s_x),(\eta\cup\{x\},\s_{\eta\cup\{x\}})\big)\rho\big((x,\s_x),(\eta,\s_\eta)\big)\\
  &=\beta\int\d x\int \Pi(\d \eta)\sum_{\s_{\eta\cup\{x\}}}\mu_{\eta\cup\{x\}}(\s_{\eta\cup\{x\}})f\big((x,\s_x),(\eta\cup\{x\},\s_{\eta\cup\{x\}})\big)\\
&=\int \Pi(\d \eta)\sum_{x\in\eta}\sum_{\s_\eta}\mu_{\eta}(\s_{\eta})f\big((x,\s_x),(\eta,\s_{\eta})\big)\\
&=\int \K(\d \bfeta) \sum_{x\in \eta}f(\bfx,\bfeta),
\end{split}
\end{equation*}
as required. 
\end{proof}

\begin{proof}[Proof of Theorem~\ref{thm_2}]
We denote again $p=p_+=p_-$. 

\medskip
{\bf Step 1:} 
As a first part, we consider the canonical Papangelou intensity given by the conditional WRM measure at the origin $o\in \R^d$ for some boundary condition. To be more precise, let $o\in \eta$ be a configuration such that the cluster of the origin in $\eta$ is finite. Then, using the independence of the WRM with respect to disjoint clusters in $\eta$, we have that
\begin{equation}\label{eq:PGG:JointCond}
\begin{split}
\frac{\mu_{\eta_\tsup{o}\cup\{o\}}(\s_{\eta_\tsup{o}}1_o)}{\mu_{\eta_\tsup{o}}(\s_{\eta_\tsup{o}})}&=p\prod_{i\in\eta_\tsup{o}\colon i\sim o}\1\{\sigma_i\neq -1\}\frac{Z_{\eta_\tsup{o}}}{Z_{\eta_\tsup{o}\cup\{o\}}},
\end{split}
\end{equation}
where again, $Z_{\eta_\tsup{o}}$ is the partition function in the definition of the WRM where the domain is only the cluster of the origin in $\eta$ but evaluated omitting $o$, and $\s_{\eta_\tsup{o}}1_o$ is the coloring of $\eta_\tsup{o}\cup\{o\}$ where $o$ is assigned the color $+1$. 

Now we can use the same construction as in the discrete part to see that the (rescaled) lattice configuration $\zeta=(1-\eps)\zeta'$, where $\eps>0$ is sufficiently small (to retain the usual lattice neighborhood structure) and $\zeta'$ is the lattice configuration from the proof of Theorem~\ref{thm_1} (there it is called $\zeta$), is a point of discontinuity for 
\begin{equation*}
\begin{split}
\eta\mapsto Z_{\eta_{\tsup{o}}}/Z_{\eta_{\tsup{o}}\cup\{o\}}.
\end{split}
\end{equation*}
The rescaling is necessary since we assumed that the connectivity threshold for the underlying PGG is $1$.

\medskip
{\bf Step 2:} Now in order to show the statement for all $p$ we can use the following trick. We consider $\zeta=(1-\eps)\zeta'$ with small $\eps$ as above and attach to each vertex in $\zeta$ a fixed finite set $\kappa\subset B_{\eps/2}(o)$ with $|\kappa|=k$, i.e., we consider the vertex-thickened lattice $\hat\zeta=\zeta\oplus\kappa$, see Figure~\ref{fig:gilbert_pt_eff} for illustration. 
\begin{figure}[ht]
\begin{minipage}{.55\textwidth}
	\begin{center}
		\includegraphics[width=\textwidth]{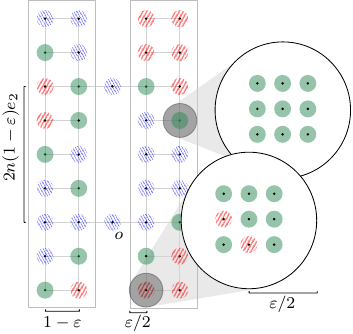}
	\end{center}	
\end{minipage}\hfill
\begin{minipage}{.4\textwidth}
	\caption{A realization of a WRM on the (rescaled) truncated half-lattices for large $p$. In black the vertices of the underlying Gilbert graph. Blue (resp.~green, red) disks represent the $+$ (resp.~$0$, $-$) spins in the WRM. The magnified areas represent the thickening procedure of Step 2, for $k=9$.}
	\label{fig:gilbert_pt_eff}
\end{minipage}
\end{figure}
We highlight that, for sufficiently small $\eps$, the points in $\hat x:=\kappa+x$ are only connected to other points in $\hat x$ and to points in $\hat y$ if $y$ is a neighbor of $x$ in $\zeta$. Note that the thickened vertex set $\hat x$ acts as a unit as, by the color constraint, all non-$0$ colored points in $\hat x$ must have the same WR-color $+1$ or $-1$. The probability that $\hat x$ has color $0$ (in the sense that all points in $\hat x$ have color $0$) under the a-priori measure is given by $(1-2p)^k$, which tends to zero as $k$ tends to infinity. Hence, for sufficiently large $k$ we enter again a regime of large (effective) $p$ as in Step 1. Reproducing Steps 2-5 in the proof of Theorem~\ref{thm_1}, but now using the effective spin $\hat\s_{\hat x}$ of the thickened vertice $\hat x$ instead of $\s_x$ gives the result.
 
\medskip
{\bf Step 3:} 
As the third step, we need to lift the discontinuity to all graph Papangelou intensities $\rho'$ that satisfy the GNZ equations with respect to $\K$. We do this here only for the case of large $p$, the general case follows by the same arguments.  
It suffices to prove that any $\bfeta$ based on the discontinuity point $\zeta=(1-\eps)\L$ with all marks set to $1$ and $0<\eps<1/2$, is also a point of discontinuity of $\bfeta'\mapsto \rho'(\bfx, \bfeta')$. 

Let us assume the contrary. Our key tool to relate $\rho'$ to $\rho$ are the GNZ equations. We want to employ them for a well-chosen sequence of functions $f_n$ that concentrate around $1_o\zeta^\tsup{1}$. For this let
\begin{equation*}
\begin{split}
f_n^{\zeta^{n,\tsup{1}}}(\bfx,\bfeta)=&\frac{1}{|B_{1/n}(o)|}\one\{\bfx=(x,\eta_x)\colon x\in B_{1/n}(o), \eta_x=1\}\\
&\times \frac{1}{\K(B_{1/n}(\zeta^{n,\tsup{1}}))}\one\{\bfeta\in B_{1/n}(\zeta^{n,\tsup{1}})\},
\end{split}
\end{equation*}
where $|B_r(x)|$ denotes the Lebesgue measure of the open ball with radius $r$ centered at $x$ and $B_r(\zeta^{n,\tsup{1}})$ denotes the event that for each $x\in \zeta^n$ there is precisely one point in $y\in \o$ at distance at most $r$ from $x$ and that $\s_y=1_x$ and additionally that $\zeta^n$ is isolated, i.e., that no other point is in distance less than $1+a$ to $\zeta^n$. Since $\K$ is based on the Poisson point processes, we have that  $\K(B_{1/n}(\zeta^{n,\tsup{1}}))>0$ for all $n\in \N$. 

Then, since we assume that $(1_o,\zeta^\tsup{1})$ is a continuity point and $\rho'$ is a graph function, we have that 
\begin{equation*}
\begin{split}
    \rho'(1_o, \bfzeta_{o^c})&=\lim_{n\uparrow\infty}\int\d \bfx\int \K(\d \bfeta)f_n^{\zeta^{n,\tsup{1}}}(\bfx,\bfeta) \rho'(\bfx,\bfeta)\\
    &=\lim_{n\uparrow\infty}\int\d \bfx\int \K(\d \bfzeta)f_n^{\zeta^{n,\tsup{1}}}(\bfx,\bfeta) \rho(\bfx,\bfeta),
\end{split}
\end{equation*}
where we used the GNZ equations twice as well as the definition of quasilocality. The same holds for $f_n^{\zeta'^{n,\tsup{1}}}$. However, since $(1_o,\zeta^\tsup{1})$ is a point of discontinuity for $\rho$, we have that 
\begin{equation*}
\begin{split}
0&=\Big|\lim_{n\uparrow\infty}\int\d \bfx\int \K(\d \bfeta)f_n^{\zeta^{n,\tsup{1}}}(\bfx,\bfeta) \rho(\bfx,\bfeta)-\lim_{n\uparrow\infty}\int\d \bfx\int \K(\d \bfeta)f_n^{\zeta'^{n,\tsup{1}}}(\bfx,\bfeta) \rho(\bfx,\bfeta)\Big|\\
&\ge \liminf_{n\uparrow\infty}|\rho(1_o,\zeta^{n,\tsup{1}})-\rho(1_o,\zeta'^{n,\tsup{1}})|\\
&\qquad -\Big|\lim_{n\uparrow\infty}\int\d \bfx\int \K(\d \bfeta)f_n^{\zeta^{n,\tsup{1}}}(\bfx,\bfeta)(\rho(\bfx,\bfeta)-\rho(1_o,\zeta^{n,\tsup{1}}))\Big|\\
&\qquad -\Big|\lim_{n\uparrow\infty}\int\d \bfx\int \K(\d \bfeta)f_n^{\zeta'^{n,\tsup{1}}}(\bfx,\bfeta)(\rho(\bfx,\bfeta)-\rho(1_o,\zeta'^{n,\tsup{1}}))\Big|,
\end{split}
\end{equation*}
where $\liminf_{n\uparrow\infty}|\rho(1_o,\zeta^{n,\tsup{1}})-\rho(1_o,\zeta'^{n,\tsup{1}})|>0$ and, 
\begin{equation*}
\begin{split}
&\limsup_{n\uparrow\infty}|\int\d \bfx\int \K(\d \bfeta)f_n^{\zeta^{n,\tsup{1}}}(\bfx,\bfeta)(\rho(\bfx,\bfeta)-\rho(1_o,\zeta^{n,\tsup{1}}))|=0
\end{split}
\end{equation*}
since, for all sufficiently large $n$, the neighborhood graph of allowed $\eta$ is precisely the same as the one given by $\zeta^n$. This is the desired contradiction.
\end{proof}

In the proofs below of the results of Section~\ref{sec:PBM}, in order to highlight the structural similarities with the previous sections, we will use the following abuse of notation, writing $\bfeta\in\Omega$ for the pair $(\eta,\s_\eta)\in\mathcal X\times\mathcal X^\pm$, $\bfy = (\bar y,\bar\sigma)\in (\R^d\times \R_+)\times \mathcal X^\pm$, and $\bfeta\cup\bfy = (\bar\eta\cup\{\bar y\},\bar\s_{\BM(\bar\eta\cup\{\bar y\})})$.

\begin{proof}[Proof of Proposition~\ref{prop_2}]
Using the Mecke theorem, we can directly calculate for all non-negative test functions $f\colon ((\R^d\times\R_+)\times \mathcal X^\pm)\times (\mathcal X\times\mathcal X^\pm) \to [0,\infty)$,
\begin{equation*}
\begin{split}
    \int&\d\bfy\int\mathbb{K}(\d\bfeta)f(\bfy,\bfeta\cup\{\bfy\})\rho(\bfy,\bfeta)\\
&=\beta\int\d\bar y\int\Pi^\pm_{\BM(\bar y)}(\d\bar\s')\int\Pi(\d\bar\eta)\frac{1}{Z_{\BM(\bar\eta\cup\{\bar y\})}}\int\Pi^\pm_{\BM(\bar\eta)}(\d\bar\s)\\
&\hspace{1.5cm}\one\{\bar\s_{\BM(\bar\eta)}\bar\s'_{\BM(\bar y)\setminus\BM(\bar y)}\text{ is feas.}\} \chi_{\bar y,\bar\eta_\tsup{y}}(\bar\s) f\big((\bar y,\bar\s'),(\bar\eta\cup\{\bar y\},\bar\s_{\BM(\bar\eta)}\bar\s'_{\BM(\bar y)})\big)\\
    &= \beta\int\d\bar y\int\Pi(\d\bar\eta)\frac{1}{Z_{\BM(\bar\eta\cup\{\bar y\})}} 
    \int\Pi^\pm_{\BM(\bar\eta\cup\{\bar y\})}(\d\bar\s)\one\{\bar\s_{\BM(\bar\eta\cup\{\bar y\})}\text{ is feas.}\} \\
&\phantom{=\beta\int\d\bfy\int\Pi(\d\bar\eta)\frac{1}{Z_{\BM(\bar\eta\cup\{\bar y\})}}\int\Pi^\pm_{\BM(\bar\eta\cup\{\bar y\})}()} f\big((\bar y,\bar\s_{\BM(\bar y)}),(\bar\eta\cup\{\bar y\},\bar\s_{\BM(\bar\eta\cup\{\bar y\})})\big)\\
    &\stackrel{\mathclap{\text{(Mecke)}}}{=}\  \int \Pi(\d\bar\eta)\frac{1}{Z_{\BM(\bar\eta)}}\int\Pi^\pm_{\BM(\bar\eta)} \one\{\bar\s_{\BM(\bar\eta)}\text{ is feas.}\} \sum_{\bar y\in\bar \eta} f(\bfy,\bfeta),
\end{split}
\end{equation*}
which verifies~\eqref{eq:PBM:GNZ_3}.
Moreover, $\rho$ is locally robust, since, for $\BM(\bar\eta_\tsup{\bar y})$ finite in $\BM(\bar\eta)$, we can find $\eps_o>0$ such that for all $\eps<\eps_o$ and $\bar\eta'\in B_\eps(\bar\eta_\tsup{\bar y},\bar\eta)$ also $\BM(\bar\eta'_\tsup{\bar y})$ bounded and in particular disconnected from the BM based on the remaining points $\bar\eta\setminus \bar\eta_\tsup{\bar y}$. In this case 
\begin{align*}
&\abs{\rho(\bar y^\tsup{1},\eta'^{\tsup{1}}) - \rho(\bar y^\tsup{1},\eta^\tsup{1})}\\
&=\beta\abs{\frac{Z_{\BM(\bar\eta'_\tsup{\bar y})}}{Z_{\BM(\bar\eta'_\tsup{\bar y}\cup\bar y)}}\, \e^{(\lambda_++\lambda_-)\abs{\BM(\bar y)\cap \BM(\bar\eta'_\tsup{\bar y})}}-\frac{Z_{\BM(\bar\eta_\tsup{\bar y})}}{Z_{\BM(\bar\eta_\tsup{\bar y}\cup\bar y)}}\, \e^{(\lambda_++\lambda_-)\abs{\BM(\bar y)\cap \BM(\bar\eta_\tsup{\bar y})}}}.
\end{align*}
Now, note that $\abs{\abs{\BM(\bar\eta'_\tsup{\bar y})}-\abs{\BM(\bar\eta_\tsup{\bar y})}}\le\abs{\BM(\bar\eta^\eps_\tsup{\bar y})\setminus\BM(\bar\eta_\tsup{\bar y})}$, 
where $\bar\eta^\eps_\tsup{\bar y}=\{( x,m_x+2\eps)\colon (x,m_x)\in\bar\eta_\tsup{\bar y}\}$, which 
tends to zero as $\eps$ tends to zero. Using this and introducing indicators for the events that no Poisson points lie in the set differences of the Boolean models, we can estimate
\begin{align*}
\abs{Z_{\BM(\bar\eta'_{\bar y})}-Z_{\BM(\bar\eta_{\bar y})}}&=\abs{\int_{\BM(\bar\eta'_{\bar y})}\Pi^\pm(\d \bar\s')\one\{\bar\s' \text{ is feas.}\}-\int_{\BM(\bar\eta_{\bar y})}\Pi^\pm(\d \bar\s)\one\{\bar\s \text{ is feas.}\}}\\
&\le (1-\e^{-\lambda_\pm\abs{\BM(\bar\eta'_{\bar y})\setminus \BM(\bar\eta_{\bar y})}})+(1-\e^{-\lambda_\pm\abs{\BM(\bar\eta_{\bar y})\setminus \BM(\bar\eta'_{\bar y})}})\\
&\qquad+\abs{\e^{-\lambda_\pm\abs{\BM(\bar\eta_{\bar y})\setminus \BM(\bar\eta'_{\bar y})}}-\e^{-\lambda_\pm\abs{\BM(\bar\eta'_{\bar y})\setminus \BM(\bar\eta_{\bar y})}}},
\end{align*}
where $\lambda_\pm=\lambda_++\lambda_-$, which tends to zero uniformly. Estimating the remaining terms in a similar fashion gives the result. 
\end{proof}

\begin{proof}[Proof of Theorem~\ref{thm:PBM:nG}]
We again denote $\lambda=\lambda_+=\lambda_-$.

\medskip
{\bf Step 1:} We show that the Papangelou intensity in \eqref{eq:PBM:pap} is not quasilocal.
As in the previous sections, the discontinuity already comes from a comparison of partition functions of the form 
\begin{align*}
  &\frac{Z_{\BM(\bar\eta_\tsup{\bar o})}}{Z_{\BM(\bar\eta_\tsup{\bar o}\cup\{\bar o\})}}
\end{align*}
for a well-chosen marked point cloud $\bar\eta$. Let us note that the additional part $\chi_{\bar x, \bar\eta_{\tsup{\bar x}}}$ may also give rise to discontinuities, but only if the radii are unbounded, something that we do not require in general. In order to see that fraction of partition functions give rise to a discontinuity, we use that the marks can be chosen arbitrarily small. We proceed as follows.
\begin{figure}
\begin{minipage}{.4\textwidth}
	\begin{center}
		\includegraphics[width=.7\textwidth]{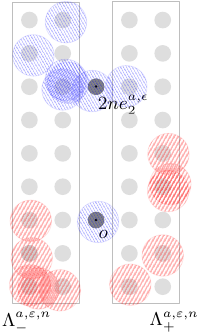}
	\end{center}	
\end{minipage}\hfill
\begin{minipage}{.6\textwidth}
	\caption{A realization of a WRM on the truncated half-lattices $\L_\pm^{a,\eps,n}$ (in grey). clusters of a PBM, represented in gray. In blue (resp.~red) the $+1$ (resp.~$-1$) marks of the WRM.}
	\label{fig:clusters_4}
\end{minipage}
\end{figure}
Consider the, by now, well-known truncated half-lattices $\L_\pm^n$ from the proof of Theorem~\ref{thm_1}. Rescaling by $2(a-\eps)$ for some small $\eps>0$ creates half-lattices $\L_\pm^{a,\eps,n}$ with the following property:

Two points $b,c$ are neighbors in $\L^{a,\eps,n}:=\L_-^{a,\eps,n}\cup \L_+^{a,\eps,n}$ if and only if, for any $x\in B_{\eps}(b)$ and $y\in B_{\eps}(c)$, we have that $B_a(x)\cap B_a(y)\neq\emptyset$. In other words, if the environment $\bar\zeta^n$ is given attaching marks $\eps$ to every point of $\zeta^n:=\L^{a,\eps,n}\cup\{o\}$, the associated Boolean model $\BM^a(\bar\zeta^n)$ has a single component. For every $x\in\zeta^n$, its ball $B_\eps(x)$ can carry no WR points, in which case it is empty, or a positive number of WR points, which then must be of the same color, see Figure~\ref{fig:clusters_4} for an illustration, where $e_2^{a,\eps}=2(a-\eps)e_2$ denotes the rescaled unit vector $e_2$. In particular, by construction, neighboring lattices points with a positive number of WR points in their $\eps$-balls force the colors of the WR points to be equal. Considering then $\zeta'^n := \zeta^n\cup\{2n e_2^{a,\epsilon}\}$ yields precisely the situation of Section~\ref{sec:DL} with $p=(1-\exp(-\lambda |B_\eps(o)|))/2$ and hence leading to the same discontinuity.

Setting all colorings of $\bar\s_{\BM(\bar\zeta)}$ to plus, and denoting such a configuration by $\zeta^\tsup{1}$, the discontinuity of $Z_{\BM(\bar\eta_\tsup{\bar o})}/Z_{\BM(\bar\eta_\tsup{\bar o}\cup\{\bar o\})}$ carries over to a discontinuity at $(\bar o^\tsup{1},\zeta^\tsup{1}\setminus\{\bar o^\tsup{1}\})$ of the map
\begin{equation}\label{eq:PBM}
    (\bfy,\bfeta) \mapsto \frac{Z_{\BM(\bar\eta_\tsup{\bar y})}}{Z_{\BM(\bar\eta_\tsup{\bar y}\cup\{\bar y\})}}\one\{(\bar\s_{\BM({\rm n}(\bar y,\bar\eta)\cup\{\bar y\})}) \text{ is feasible}\}.
\end{equation}

\medskip
{\bf Step 2:} We lift the discontinuity from Step 1 to all (locally robust) Papangelou intensities $\rho'$ that satisfy~\eqref{eq:PBM:GNZ_3} with respect to $\mathbb K$. 

It suffices to prove that $\zeta^\tsup{1}$ is also a point of discontinuity for $\rho'$. 
Let us assume the contrary. 
We want to employ~\eqref{eq:PBM:GNZ_3} to connect $\rho'$ and $\rho$ via the following test function
\begin{align*}
    f_n^{\zeta^{n,\tsup{1}}}(\bfeta)&=\frac{\one\{\exists \bar y\in\bar\eta\colon y\in B_{1/n}(o), m_{y}\in B_{1/n}(\eps)\}}{|B_{1/n}(o)|\nu(B_{1/n}(\eps))}\frac{\one\{\exists \bar x\in\bar\s_{\BM(\bar y)}\colon x\in B_{1/n}(o),\s_{x}=1\}}{|B_{1/n}(o)|}\\
    &\qquad\times\frac{\one\{\bar\eta\in B_{1/n}(\bar\zeta^n)\}}{\Pi(B_{1/n}(\bar\eta^n))}\frac{\one\{\bar\s\in B_{1/n}(\bar\s^n)\}}{\Pi^\pm_{\BM(\bar\eta')}(B_{1/n}(\bar\s^n))},
\end{align*}
where $\bar\s^n$ is some reference configuration compatible with $\bar\zeta^n$ in which all marks are set to $+$. 
Heuristically, $f_n^{\zeta^{n,\tsup{1}}}$ assigns positive weight only to those configurations that are close to $\zeta^{n,\tsup{1}}$; all indicators are compensated by the corresponding probability weights under $\mathbb K$ and the Lebesgue and mark measures, making $f_n^{\zeta^{n,\tsup{1}}}$ a probability density. 

Now, using the fact that $\rho'$ is assumed to be quasilocal, we have that
\begin{equation*}
\begin{split}
    \rho'(\bfo,\bfeta\setminus \bfo)
    &=\lim_{n\uparrow\infty}\int\d \bfy \int \mathbb K(\d\bfeta)f_n^{\zeta^{n,\tsup{1}}}(\bfeta\cup\{\bfy\}) \rho'(\bfy,\bfeta)\\
    &=\lim_{n\uparrow\infty}\int\d \bfy \int \mathbb K(\d\bfeta)f_n^{\zeta^{n,\tsup{1}}}(\bfeta\cup\{\bfy\}) \rho(\bfy,\bfeta)=\rho(\bfo, \bfeta\setminus \bfo).
\end{split}
\end{equation*}
However, $\zeta^\tsup{1}$ is a point of discontinuity for $\rho$ and hence also for $\rho'$. This is the desired contradiction. 

\medskip
{\bf Step 3:} In order to see that for sufficiently small $\lambda_+=\lambda_-$ we reenter a regime of quasilocality, first note that, under the bounded-radii condition, the function $\bar\eta\mapsto\chi_{\bar x, \bar\eta_\tsup{\bar x}}$ is quasilocal for all $\bar x$ since it is local. Furthermore, since also the indicator $\one\{\bar\s'_{\BM(\bar x)}\bar\s_{\BM({\rm n}(\bar x,\bar \eta))\setminus \BM(\bar x)} \text{ is feasible}\}$ is quasilocal under the bounded-radii condition, it suffices to show the continuity of the ratio of partition functions $\bfeta\mapsto Z_{\BM(\bar\eta_\tsup{\bar o})}/Z_{\BM(\bar\eta_\tsup{\bar o}\cup\{\bar o\})}$. 

For this, as in Step 7 of the proof of Theorem~\ref{thm_1}, note that we can write the ratio $Z_{\BM(\bar\eta_\tsup{\bar o}\cup\{\bar o\})}/Z_{\BM(\bar\eta_\tsup{\bar o})}$ as an expectation of a local function with respect to the WRM $\mu_{\BM(\bar\eta_\tsup{\bar y})}$. Now, the random-cluster representation of the symmetric WRM is dominated by a Poisson--Boolean model with intensity $\lambda_++\lambda_-$ and the same interaction threshold $a$. More precisely, note that the WR density (i.e., the normalized feasibility indicator) is decreasing with respect to the addition of points. On the other hand, the percolation event of existence of a connected component that connects $\BM(\bar y)$ with $B_{r}^{\rm c}(o)$, for $r>0$ large and on $\BM(\bar\eta_\tsup{\bar y})$, is increasing under the addition of points. Thus, using the FKG inequality with respect to the PPP with intensity $\lambda_++\lambda_-$ (that underlies the random-cluster representation), the probability of the percolation event under the random-cluster representation of the WRM is bounded from above by the probability of the percolation event under the (unconstrained) PPP with the same intensity, see~\cite[Appendix]{JK17}. 

Furthermore, if $\lambda_++\lambda_-$ is sufficiently small, the probability of the percolation event under the PPP tends to zero as $r$ tends to infinity on the whole domain $\R^d$ and hence also on all subdomains of the form $\BM(\bar\eta_\tsup{\bar y})$. But this implies absence of percolation in the random-cluster representation of the WRM also on any subdomain and hence the WRM model is in the uniqueness regime. This implies quasilocality.
\end{proof}

\section{Appendix}\label{sec:ap}
The formalism used in Section~\ref{sec:PBM} is convenient since it allows us to present a canonical version of the Papangelou intensity in a simple form corresponding well to the canonical Papangelou intensities of the other models in this manuscript. However, the resulting process under $\mathbb K$ is not a marked point process, since every point $x$ carries an iid radius mark $m_x$, but then a whole cluster $\BM(\bar\eta^\ssup{k})$ carries a joint mark $\bar\s_{\BM(\bar\eta^\ssup{k})}$ which is not clearly coming from the individual points in $\eta^\ssup{k}$. 

\subsection{Removal of overlaps}
\begin{figure}[h]
\begin{minipage}{.5\textwidth}
    \begin{center}
	\includegraphics[width=.8\textwidth]{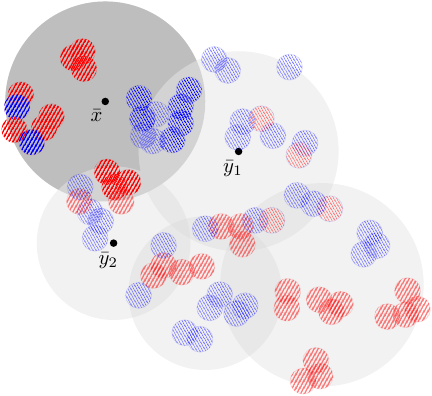}
    \end{center}	
\end{minipage}\hfill
\begin{minipage}{.5\textwidth}
    \begin{center}
	\includegraphics[width=.8\textwidth]{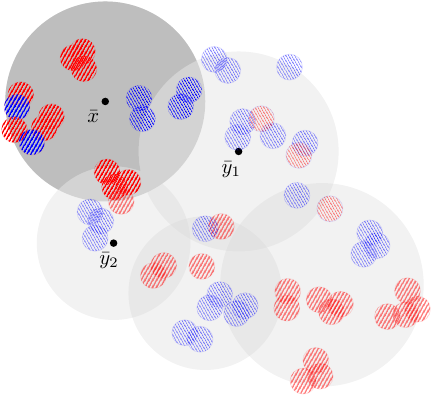}
    \end{center}
\end{minipage}  
\caption{Colored points on the cluster of a PBM (in gray). On the left, a realization of $\bfPi$ where overlapping regions have too many points; on the right a realization of $\bfPi^\ssup{1}$. Note the different intensities in the overlap regions.}
    \label{fig:PBM:removal}
\end{figure}

We start by constructing a suitable reference measure based on an iid-marked PPP that includes the colored points but has no color-dependent interaction. There is however the necessity to compensate the possible overcounting of colored points in overlapping environment regions. To be more precise, recall the a-priori measure introduced in Section~\ref{sec:PBM}, which is given by an iid marked  PPP $\bfeta$ where we attach to every point $x\in\eta$ the mark $(m_x,\bar\s_{\ssup x})\in\mathcal M=\R_+\times\mathcal X^\pm$. In words, the mark consists of the radius $m_x$ and a two-color marked Poisson point cloud $\bar\s=\bar\s_\ssup{x}-x$. The marks are iid distributed via the probability measure $\nu(\d m)\Pi^\pm_{\BM((o,m))}(\d\bar\s)$ and we write $\bfx$ for a marked point in the iid marked PPP $\bfeta$ with distribution $\bfPi(\d\bfeta)$.
If the balls do not overlap, the configuration of WR-particles follows the associated distribution in $\mathbb K$. Otherwise, it creates
dependencies between the marks 
that have to be taken care of. 
Note that the dependencies are due to non-empty overlaps $B_{m_x}(x)\cap B_{m_y}(y)$ between the balls associated to two points $x,y\in \eta$ since, in the overlap, the independent a-priori measure places too many colored particles. More precisely, in the overlap $B_{m_x}(x)\cap B_{m_y}(y)$, the colored particles (coming from the WR-part of the marks of $x$ and $y$) are still a (two-color marked) Poisson point process, but now with twice the intensity. More generally, the union of the colored points $\bigcup_{x\in \eta}\bar\s_\ssup{x}$ has too many points in non-empty overlapping regions, compared to the a-priori measure underlying~\eqref{eq:PBM:joint}, see also Figure~\ref{fig:PBM:removal}. This phenomenon does not occur in the model of Section~\ref{sec:PGG} as environment points do not overlap there. In particular, in view of the GNZ approach, adding a point 
$\bfx=(x,(m_x,\bar\s_\ssup{x}))=(\bar x,\bar\s_\ssup{x})$ to a configuration $\bfeta$,
i.e., considering $\bfeta\cup\{\bfx\}$, would also potentially lead to too many points in $\BM(\bar x)\cap \BM(\bar\eta)$ (in addition to them also having to satisfy the WR constraint).
In order to resolve this overlap issue, the question is how to decide to which environment point a WR-particle 
drawn from the given joint measure in an overlap of two or more environment balls will be associated? Conversely, 
how to choose compound mark intensities such that they reproduce the given joint measure? For example, one can devise a random algorithm to remove all but one configuration of marks in the overlapping region:
To each point $x$ of the environment configuration $\eta$, we assign an additional \emph{choice variable} given by a uniform random variable $w_x\in [0,1]$; whenever two (or more) balls overlap the colored WR-points in the overlap area are assigned to the environment point with the lower choice variable. For example, if $B_{m_x}(x)\cap B_{m_y}(y)\neq \emptyset$, then the colored points are assigned to $x$ if $w_x<w_y$.
We now detail this procedure.
Let us define 
\begin{equation*}
\bfPi^\ssup{1}(\d\bfeta):=\bfPi(\d\bfeta)\chi(\bfeta),
\end{equation*}
where $\chi(\bfeta):=\prod_{x\in\eta}\chi(\bfx,\bar\eta)$ with
\begin{equation}\label{eq:PBM:joint3}
\begin{split}
\chi(\bfx,\bar\eta):=\int \Leb(\d w_{\eta\cup\{x\}})\one\{(\s_{\ssup{x}}^+\cup\s_{\ssup{x}}^-)\cap (R((\bar x,w_x),(\bar\eta,w_\eta))-x)=\emptyset\}\\
\times T^{-1}\big((\bar x,w_x),(\bar\eta,w_\eta)\big).
\end{split}
\end{equation}
Here,  $\Leb(\d w_{\eta\cup\{x\}})=\bigotimes_{y\in \eta\cup\{x\}}\Leb_{[0,1]}(\d w_y)$ is the independent superposition of uniform distributions on $[0,1]$ (which serve as choice variables) and 
$$R((\bar x,w_x),(\bar\eta,w_\eta))=\BM(\bar x)\cap \bigcup_{y\in\eta\colon w_y<w_x}\BM(\bar y)$$ 
represents the intersection of the ball $B_{m_x}(x)$ with all other balls $B_{m_y}(y)$ in $\eta$ corresponding to points with smaller choice variables. This construction ensures that in any region of $\BM(\bar\eta)$ we remove all but precisely one coloring mark, the coloring associated to the Poisson point with smallest choice variable, see Figure~\ref{fig:PBM:removal}. The normalization 
$$T\big((\bar x,w_x),(\bar\eta,w_\eta)\big):=\Pi^\pm\big((\s_{\ssup{x}}^+\cup\s_{\ssup{x}}^-)\cap R((\bar x,w_x),(\bar\eta,w_\eta))=\emptyset\big)$$
 in~\eqref{eq:PBM:joint3} is almost-surely finite by our subcriticality assumption and it further ensures that $\bfPi^\ssup{1}$ is still a probability measure. From this construction it follows that almost-surely under $\bfPi^\ssup{1}$, the configuration $\bigcup_{x\in\eta}\bar\s_{\ssup{x}}$ has an intensity precisely given by $\lambda_\pm\one\{y\in\BM(\bar\eta)\}\d y$. 

\subsection{The GNZ equations}
Now, we incorporate the WRM constraint, see Figure~\ref{fig:PBM:pap2}, into the previously defined a-priori measure $\bfPi^{\ssup{1}}$ via
\begin{equation*}
\mathbb K^\ssup{1}(\d\bfeta):=\bfPi^\ssup{1}(\d \bfeta)\bar\mu(\bfeta),
\end{equation*}
where $\bar\mu(\bfeta)=\prod_{k\in I}\bar\mu(\bfeta^{\ssup{k}})$ is a product over the clusters in $\BM^a(\bar\eta)$, with 
\begin{equation*}
\bar\mu(\bfeta^{\ssup{k}})=\one\Big\{\bigcup_{x\in\eta^\ssup{k}}\bar\s_{\ssup{x}}\text{ is feasible}\Big\}Z^{-1}(\bar\eta^\ssup{k})
\end{equation*}
with $Z(\bar\eta^\ssup{k})=\int\Pi^\pm_{\BM(\bar\eta^\ssup{k})}(\d\bar\s)\one\big\{\bar\s\text{ is feasible}\big\}$. This verifies Lemma~\ref{lem:Model3}.

\begin{figure}
\begin{minipage}{.45\textwidth}
    \begin{center}
	\includegraphics[width=.9\textwidth]{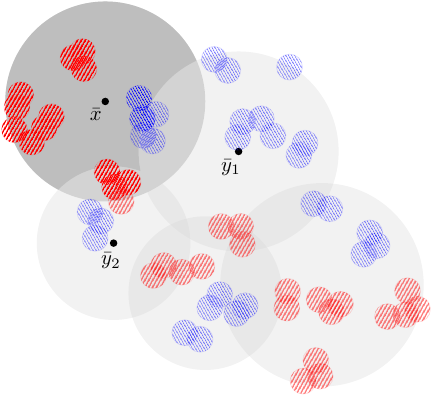}
    \end{center}	
\end{minipage}\hfill
\begin{minipage}{.5\textwidth}
    \caption{A realization of $\mathbb K^{\ssup{1}}$: a WRM on the clusters of a PBM, represented in gray. The blue ($+1$) and red ($-1$) marks are subject to the WR constraint.}
    \label{fig:PBM:pap2}
\end{minipage}    
\end{figure}

\medskip
Let us finally verify the GNZ equation for $\mathbb K^{\ssup{1}}$.
\begin{proposition}\label{prop_3}
The measure $\mathbb K^{\ssup{1}}$ satisfies the GNZ equations with respect to the Papangelou intensity
\begin{equation}\label{eq:PBM:pap3}
\begin{split}
\rho(\bfx,\bfeta)=\beta\frac{\chi(\bfeta\cup\{\bfx\})\bar\mu(\bfeta\cup\{\bfx\})}{\chi(\bfeta)\bar\mu(\bfeta)}\one\{\bar\eta_{\tsup{x}}\text{ is finite}\}.
\end{split}
\end{equation}
\end{proposition}
The difference between the expressions~\eqref{eq:PBM:pap3} and~\eqref{eq:PBM:pap} lies in the handling of the overlaps. The color constraint encoded in the ratio of the $\bar\mu$, respectively the ratio of the partition functions $Z$, is precisely the same and is the one responsible for the possible lack of continuity. Finally note that $\bar\mu$ as well as $\chi$ factorize over connected components in $\BM^a(\bar\eta)$ and hence in the ratio in~\eqref{eq:PBM:pap3} we see cancellations. 
\begin{proof}[Proof of Proposition~\ref{prop_3}]
Using the Mecke formula for the iid marked PPP $\bfPi$, we can calculate for all non-negative measurable test functions $f$,
\begin{equation*}
\begin{split}
    \int\d\bfx&\int\mathbb{K}^{\ssup{1}}(\d\bfeta)\rho(\bfx,\bfeta)f(\bfx,\bfeta\cup\{\bfx\})=\int\d\bfx\int\bfPi(\d\bfeta)\chi(\bfeta)\bar\mu(\bfeta)\rho(\bfx,\bfeta)f(\bfx,\bfeta\cup\{\bfx\})\\
&=\beta\int\d\bfx\int\bfPi(\d\bfeta)\chi(\bfeta\cup\{\bfx\})\bar\mu(\bfeta\cup\{\bfx\})f(\bfx,\bfeta\cup\{\bfx\})\\
&=\int\bfPi(\d\bfeta)\chi(\bfeta)\bar\mu(\bfeta)\sum_{x\in \eta}f(\bfx,\bfeta)\\
&=\int\mathbb{K}^{\ssup{1}}(\d\bfeta)\sum_{x\in \eta}f(\bfx,\bfeta),
\end{split}
\end{equation*}
which verifies the GNZ equation.
\end{proof}

\paragraph{\bf Acknowledgements.}
BJ received support by the Leibniz Association within the Leibniz Junior Research Group on \textit{Probabilistic Methods for Dynamic Communication Networks} as part of the Leibniz Competition (grant no.\ J105/2020). The research of BJ and AZ was funded by Deutsche Forschungsgemeinschaft (DFG) through DFG Project no.~422743078.
BJ is also funded through DFG Project no.~531531795.

\bibliographystyle{alpha}
\bibliography{references.bib}

\end{document}